\documentclass[reqno]{amsart}
\usepackage{graphics}
\usepackage{latexsym}
\usepackage{amsfonts}
\usepackage{amsmath}
\usepackage{amssymb}
\usepackage{amsthm}
\usepackage{multicol}
\usepackage{dsfont}
\usepackage{enumerate}
\numberwithin{equation}{section}
\newtheorem{theorem}{Theorem}[section] 
\newtheorem{proposition}{Proposition}[section] 
\newtheorem{lemma}{Lemma}[section] 
\newtheorem{remark}{Remark}[section] 

\newcommand{\R}{\mathbb R} 
\newcommand{\C}{\mathbb C} 
\newcommand{\N}{\mathbb N} 

\renewcommand{\P}{\mathbb P}
\newcommand{\E}{\mathbb E}
\newcommand{\V}{\mathbb V}
\newcommand{\Cov}{\mathbb Cov}

\newcommand{\tr}{\text{tr}}
\newcommand{\Tr}{\text{Tr}}

\renewcommand\Re{\operatorname{\mathfrak{Re}}}
\renewcommand\Im{\operatorname{\mathfrak{Im}}}

\begin{document}
\title[Finite Rank Deformations]{On Finite Rank Deformations of Wigner Matrices}
\author[A. Pizzo]{Alessandro Pizzo}
\address{Department of Mathematics, University of California, Davis, One Shields Avenue, Davis, CA 95616-8633  }
\thanks{A.P. has been supported in part by the NSF grant DMS-0905988}
\email{pizzo@math.ucdavis.edu}
\author[D. Renfrew]{David Renfrew} \thanks{D.R. has been supported in part by the NSF grants VIGRE DMS-0636297, DMS-1007558, and 
DMS-0905988 }
\address{Department of Mathematics, University of California, Davis, One Shields Avenue, Davis, CA 95616-8633  }
\email{drenfrew@math.ucdavis.edu}
\author[A. Soshnikov]{Alexander Soshnikov}
\address{Department of Mathematics, University of California, Davis, One Shields Avenue, Davis, CA 95616-8633  }
\thanks{A.S. has been supported in part by the NSF grant DMS-1007558}
\email{soshniko@math.ucdavis.edu}
\begin{abstract}
We study the distribution of the outliers in the spectrum of finite rank deformations of Wigner random matrices
under the assumption that the absolute values of the off-diagonal matrix entries have uniformly bounded fifth moment and the absolute values of the 
diagonal entries have uniformly bounded third moment.  Using our recent results on the fluctuation of resolvent entries
\cite{PRS}, \cite{ORS} and ideas from \cite{BGM}, we extend
the results by Capitaine, Donati-Martin, and F\'eral \cite{CDF}, \cite{CDF1}.
\end{abstract}
\maketitle
\section{ \bf{Introduction and Formulation of Main Results}}
\label{sec:intro1}
Let $X_N= \frac{1}{\sqrt{N}} W_N $ be a random real symmetric (Hermitian) Wigner matrix with independent entries up from the diagonal.
In the real symmetric case, we assume that the off-diagonal entries 
\begin{equation}
\label{offdiagreal1}
(W_N)_{ij},\ 1\leq i<j\leq N,
\end{equation}
are independent random variables such that
\begin{equation}
\label{offdiagreal2}
\E (W_N)_{ij}=0, \ \V(W_N)_{ij}=\sigma^2, \  m_5:=\sup_{i\not=j, N} \E |(W_N)_{ij}|^5 <\infty,
\end{equation}
where $\E \xi $ denotes the mathematical expectation and $\V \xi $ the variance of a random variable $\xi.$
The diagonal entries  
\begin{equation}
\label{diagreal1}
(W_N)_{ii}, \ 1\leq i \leq N,
\end{equation}
are independent random variables (that are also independent from the off-diagonal entries), such that
\begin{equation}
\label{diagreal2}
\E (W_N)_{ii}=0, \ c_3:=\sup_{i,N} \E |(W_N)_{ii}|^3<\infty.
\end{equation}

In a similar fashion, in the Hermitian case, we assume that the off-diagonal entries
\begin{equation}
\label{offdiagherm1}
\Re (W_N)_{ij}, \Im (W_N)_{ij}, \ 1\leq i<j \leq N,
\end{equation}
are independent random variables such that
\begin{equation}
\label{offdiagherm2}
\E (W_N)_{ij}=0, \ \V [\Re (W_N)_{ij}]=\V [\Im (W_N)_{ij}]=\frac{\sigma^2}{2}, \  m_5:=\sup_{i\not=j, N} \E |(W_N)_{ij}|^5 <\infty.
\end{equation}
The diagonal entries  
\begin{equation}
\label{diagherm1}
(W_N)_{ii}, \ 1\leq i \leq N,
\end{equation}
are independent centered random variables, independent from the off-diagonal entries, with uniformly bounded  third moment of the absolute values.

For a  real symmetric (Hermitian) matrix $M$ of order $N,$ its empirical distribution of the eigenvalues is defined as 
$\mu_M = \frac{1}{N} \sum_{i=1}^{N} \delta_{\lambda_{i}},$ where  $\lambda_1 \geq \ldots \geq \lambda_N$ are the (ordered) eigenvalues of $M.$
Wigner semicircle law (see e.g. \cite{BG}, \cite{AGZ}, \cite{B}) states 
that almost surely the empirical distribution $\mu_{X_N}$ of a random real symmetric (Hermitian) Wigner matrix $X_N$ converges weakly 
to the nonrandom limiting distribution $\mu_{sc}.$ The limiting distribution 
$\mu_{sc} $ is known as the semicircle distribution.  It is absolutely continuous with respect to the Lebesgue measure and has 
the compact support $[-2 \sigma , 2 \sigma].$  The density of the Wigner semicircle distribution
is given by 
\begin{equation}
\label{polukrug}
\frac{d \mu_{sc}}{dx}(x) = \frac{1}{2 \pi \sigma^2} \sqrt{ 4 \sigma^2 - x^2} \mathbf{1}_{[-2 \sigma , 2 \sigma]}(x).
\end{equation}
Wigner semicircle law can be reformulated as follows.
For any bounded continuous test function $\varphi: \R \to \R,$ the linear statistic 
\[ \frac{1}{N} \sum_{i=1}^N \varphi(\lambda_i) = \frac{1}{N} \*\Tr(\varphi(X_N))=:\tr_N(\varphi(X_N)) \]
converges to $\int \varphi(x) \* d \mu_{sc}(dx) $ almost surely; here and throughout the paper, we 
use the notation $\tr_N = \frac{1}{N} \Tr$ to denote the normalized trace.

The Stieltjes transform of the semicircle law is
\begin{equation}
\label{steltsem} 
g_\sigma(z) := \int \frac{d \mu_{sc}(x)}{z-x}= \frac{z-\sqrt{z^2-4\*\sigma^2}}{2\*\sigma^2}, \ z \in \C \backslash [-2\*\sigma, 2\*\sigma].
\end{equation}
It is the solution to 
\begin{equation}
\label{semicircle}
\sigma^2 g_\sigma^2(z) - z g_\sigma(z) + 1 = 0
\end{equation}
that decays to 0 as $|z| \to \infty$. 

In this paper, we study the fluctuations of the outliers in the spectrum of finite-dimensional deformations of Wigner matrices.
Starting with the pioneering work by F\"uredi  and Koml\'os \cite{FK}, there have been several results on finite rank 
perturbations of matrices with i.i.d. entries, in particular \cite{P}, \cite{FP}, \cite{Mai}, \cite{CDF}, \cite{CDF1}, \cite{CDF2}, \cite{BR}, 
\cite{BGM}, \cite{BGM1}, \cite{T}. We also note several papers on the eigenvalues of sample covariance matrices of spiked population models 
(\cite{BY}, \cite{BBP}, \cite{BS}, \cite{Paul}).  

This manuscript can be viewed as a companion paper to our recent works \cite{PRS} and \cite{ORS} on the non-Gaussian 
fluctuation of the matrix entries of regular functions of Wigner matrices.  However, no knowledge of the machinery used in \cite{PRS} and \cite{ORS} 
is required, and the paper can be read independently from these papers.

Let us consider a deformed Wigner matrix 
\begin{equation}
\label{MN}
M_N := \frac{1}{\sqrt{N}} W_N + A_N = X_N + A_N.
\end{equation}
Here $W_N$ is a random real symmetric (Hermitian) Wigner matrix as defined in (\ref{offdiagreal1}-\ref{diagreal2}) ((\ref{offdiagherm1}-\ref{diagherm1})),
and $A_N$ is a deterministic Hermitian 
matrix of fixed finite rank $r.$  
We assume that the eigenvalues of $A_N$ and their multiplicities are fixed.  
Let 
$$\theta_1 > \ldots > \theta_{j_0}=0 > \ldots  > \theta_J$$  
be the eigenvalues 
of $A_N$ each with fixed multiplicity $k_j$.  Clearly, the eigenvalue $\theta_{j_0}=0 $ has multiplicity $N-r$ and 
$ \sum_{j\neq j_0} k_j =r.$  

The first theorem of this section, Theorem \ref{thm:outliers}, concerns the convergence of the extreme eigenvalues of the deformed random matrix.  
Let us denote $\rho_\theta = \theta + \frac{\sigma^2}{\theta}. $  
We shall use the shorthand notation $ \rho_j $ for $\rho_{\theta_j}. $
Theorem \ref{thm:outliers} was originally proved by Capitaine, Donati-Martin, and Feral in \cite{CDF} in the case when the common marginal 
distribution of the matrix entries is symmetric and 
satisfies a Poincar\'e inequality.  
\begin{theorem} 
\label{thm:outliers}
Let $X_N=\frac{1}{\sqrt{N}} W_N$ be a random real symmetric (Hermitian) Wigner matrix
satisfying (\ref{offdiagreal1}-\ref{diagreal2}) (respectively (\ref{offdiagherm1}-\ref{diagherm1})).
Let $J_{\sigma^+}$ be the number of $j$'s such that $\theta_j > \sigma$ and  $J_{\sigma^-}$ be the number of $j$'s such that $\theta_j < -\sigma$.
\begin{enumerate}[(a)]
\item For all $j = 1, \ldots, J_{\sigma^+}$ and $i = 1, \ldots, k_j$, $\lambda_{k_1 + \ldots + k_{j-1}+i} \to \rho_j,$ 
\item $\lambda_{k_1 + \ldots + k_{J_{\sigma^+}}+ 1} \to 2 \sigma,$
\item $\lambda_{k_1 + \ldots + k_{J - J_{\sigma^- }} } \to -2 \sigma,$
\item For all $j = J- J_{\sigma^-}+1 , \ldots, J$ and $i = 1, \ldots, k_j$, $\lambda_{k_1 + \ldots + k_{j-1}+i} \to \rho_j.$
\end{enumerate}
The convergence is in probability.  
\end{theorem}
In other words, the first $k_1$ largest eigenvalues of $M_N$ converge to $\rho_1, $ the next $k_2$ largest eigenvalues converge 
to $\rho_2, \ldots,$ the $J_{\sigma^+}$th bunch of the largest eigenvalues converge to $\rho_{J_{\sigma^+}},$   the next largest 
eigenvalue converges to $2\*\sigma$ (since it corresponds to a nonnegative eigenvalue of $A_N$ which is not bigger than $\sigma$), etc.

\begin{remark}
If random variables $(W_N)_{ij}, \ 1\leq i\leq j\leq N,$ satisfy a  Poincar\'e inequality (\ref{poin}) with constant $\upsilon_{i,j,N}$ uniformly bounded
from zero, $\upsilon_{i,j,N}\geq \upsilon>0,$ the convergence holds with probability one.
\end{remark}
We recall that a probability measure  $\P$ on $\R^M$ satisfies a Poincar\'e inequality with constant $\upsilon>0$ if, for all continuously differentiable 
functions 
$f: \R^M \to \C,$
\begin{equation}
\label{poin}
\V_{\P}(f)= \E_{\P} \left(|f(x)-\E_{\P}(f(x))|^2\right) \leq \frac{1}{\upsilon}\* \E_{\P}[ |\nabla f(x)|^2 ]
\end{equation} 
Note that the Poincar\'e inequality tensorizes and the probability measures satisfying the Poincar\'e inequality
have subexponential tails (\cite{GZ}, \cite{AGZ}) . In particular, if the marginal distributions of the matrix entries of $W_N$
satisfy the Poincar\'e inequality with constant $\upsilon>0,$ then the joint distribution of $(W_N)_{ij}, \ 1\leq i\leq j\leq N,$ also satisfies
the Poincar\'e inequality with the same constant $\upsilon.$
By a standard scaling argument, we note that if the marginal distributions of the matrix entries of $W_N$
satisfy the Poincar\'e inequality with 
$\upsilon>0$ then the marginal distributions of the matrix entries of $X_N=\frac{1}{\sqrt{N}}\*W_N $ satisfy the Poincar\'e inequality with constant
$N\*\upsilon.\ $  

Theorem \ref{thm:outliers} follow from Theorem \ref{thm:fluctoutliers} formulated below.
Theorem \ref{thm:fluctoutliers} is concerned with the distribution of the outliers, i.e. the eigenvalues of $M_N$ corresponding to 
$\theta_j >\sigma.$  Namely, we are interested in the fluctuation of the outliers around $\rho_j, \ 1 \leq j \leq J_{\sigma^+}. $
Let us consider  a fixed eigenvalue
$\theta_j $ of $A_N$ such that $\theta_j> \sigma.$
In general, if one does not assume some additional information about the structure of the eigenvectors of $A_N$ corresponding to $\theta_j,$ 
the sequence of random vectors
\begin{equation}
\label{ura}
\Big( \sqrt{N}\*(\lambda_{k_1+\ldots +k_{j-1}+i} -\rho_j), \ i=1, \ldots, k_j \Big) 
\end{equation}
does not converge in distribution as $N \to \infty$  (see Theorems \ref{thm:caseA} and \ref{thm:caseB} below).
However, one can prove that the sequence (\ref{ura}) is bounded in probability (tight). We recall 
that a sequence $\{\xi_N\}_{N\geq 1} $ of $\R^m$-dimensional random vectors is bounded in probability if for any $\varepsilon>0 $ there 
exists $L(\varepsilon)$ that does not depend on $N$ such that $\P(|\xi_N|>L(\varepsilon))<\varepsilon $ for all $N\geq 1.$
\begin{theorem} 
\label{thm:fluctoutliers}
Let $X_N=\frac{1}{\sqrt{N}} W_N$ be a random real symmetric (Hermitian) Wigner matrix
defined in (\ref{offdiagreal1}-\ref{diagreal2}) (respectively (\ref{offdiagherm1}-\ref{diagherm1})).
Let $1 \leq j \leq J_{\sigma^+},$ so
the eigenvalue $\theta_j $ of $A_N$ satisfies $\theta_j> \sigma$.
Then the sequence of random vectors 
\begin{equation}
\label{past}
\Big( \sqrt{N}\*(\lambda_{k_1+\ldots +k_{j-1}+i} -\rho_j), \ i=1, \ldots, k_j \Big) 
\end{equation}
is bounded in probability.
In addition, if the marginal distributions of the matrix entries of $W_N$
satisfy the Poincar\'e inequality (\ref{poin}) with constant $\upsilon_{i,j,N}$ uniformly bounded
from zero, the following holds with probability $1$
\begin{equation}
\label{smeshno}
\lambda_{k_1+\ldots +k_{j-1}+i} -\rho_j = O\left(\frac{\log N}{\sqrt{N}}\right),  \ i=1, \ldots, k_j.
\end{equation}
\end{theorem}

\begin{remark}
Theorem \ref{thm:fluctoutliers} clearly implies parts (a) and (d) of Theorem \ref{thm:outliers}.  To see that parts (b) and (c) of 
Theorem \ref{thm:outliers} also follow, we note that for any fixed positive integer $l\geq 1$ the $l$-th largest eigenvalue of $X_N$ converges in 
probability
to $2\*\sigma.$   This is a simple consequence of the convergence of the largest eigenvalue of $X_N$ to $2\*\sigma$ and the semicircle law.  Then the 
interlacing property and Theorem \ref{thm:fluctoutliers} imply the desired result.
\end{remark}

\begin{remark}
The bound (\ref{smeshno}) means that there exists a sufficiently large deterministic constant $C=C(\sigma, \upsilon, \theta_1, \ldots, \theta_r)>0, $ 
such that with probability $1$
\begin{equation*}
|\lambda_{k_1+\ldots +k_{j-1}+i} -\rho_j|\leq \frac{C\*\log N}{\sqrt{N}},  \ i=1, \ldots, k_j,
\end{equation*}
for all but finitely many $N.$
\end{remark}
To study the fluctuations of the outliers in more detail, we consider two special cases following \cite{CDF1}.

{\bf Case A} (``The eigenvectors don't spread out'')

The orthonormal eigenvectors of $A_N$ corresponding to $\theta_j$ depend on a finite number $K_j$ of canonical basis vectors of $\C^N$ (without loss 
of generality we can assume those canonical vectors to be $e_1, \ldots, e_{K_j}$), and their 
coordinates are independent of $N.$

{\bf Case B} (``The eigenvectors are delocalized'')

The $l^{\infty}$ norm of every orthonormal eigenvector of $A_N$ corresponding to $\theta_j$ goes to zero as $N \to \infty. $

Following \cite{CDF1}, we denote 
by $k_{\sigma^+}= k_1 +\ldots +k_{J_{\sigma^+}}$ the number of positive eigenvalues of $A_N$ bigger than $\sigma$ (counting with multiplicities) and
by $k\geq k_{\sigma^+} $ the minimal number of canonical basis vectors $e_1, \ldots, e_N$ of $\C^N$ required to span all the eigenvectors 
corresponding to the eigenvalues $\theta_1, \ldots, \theta_{J_{\sigma^+}}. $  

We also denote
\begin{equation}
\label{ctheta}
c_{\theta_j}:= \frac{\theta_j^2}{\theta_j^2-\sigma^2}.
\end{equation}
The next theorem is a consequence of Proposition \ref{Prop1} below and Theorems 1.1 and 1.5 in \cite{PRS}. We use a standard notation 
$\beta=1$ in the real symmetric case and $\beta=2$ in the Hermitian case.
\begin{theorem} 
\label{thm:caseA}
Let $X_N=\frac{1}{\sqrt{N}} W_N$ be a random real symmetric (Hermitian) Wigner matrix
defined in (\ref{offdiagreal1}-\ref{diagreal2}) (respectively (\ref{offdiagherm1}-\ref{diagherm1})) such that the off-diagonal entries
$(W_N)_{ij}, \ 1\leq i<j\leq N,$ are i.i.d. real (complex) random variables with probability distribution $\mu$ and the diagonal entries 
$(W_N)_{ii}, \ 1\leq i< N,$ are i.i.d. random variables with probability distribution $\mu_1.$
In Case A, the $k_j$-dimensional vector 
\[ \Big( c_{\theta_j}\* \sqrt{N}\*(\lambda_{k_1+\ldots +k_{j-1+i}} -\rho_j), \ i=1, \ldots, k_j \Big) \]
converges in distribution to the distribution of the ordered eigenvalues of the $k_j \times k_j$ random matrix $V_j$ defined as
\begin{equation}
\label{vj}
V_j:=U^*_j \* (W_j + H_j) \* U_j,
\end{equation}
where 

(i) $W_j$ is a Wigner random matrix of size $K_j$ with the same marginal distribution of the matrix entries as $W_N,$ 

(ii) $H_j$ is a real symmetric (Hermitian) 
Gaussian matrix of size $K_j,$ independent of $W_j,$ with centered independent entries $H_{st}, \ 1 \leq s \leq t \leq K_j,
(\Re H_{st}, \Im H_{st}, \ 1 \leq s <t \leq K_j, \ H_{pp}, 1 \leq p \leq K_j$ in the Hermitian case)
with
the variance of the entries given by
\begin{align}
\label{ovec1}
& \E(H^2_{ss})=\frac{\kappa_4(\mu)}{\theta_j^2} + 
\frac{2}{\beta}\* \frac{\sigma^4}{\theta_j^2-\sigma^2}, \ s=1, \ldots, K_j, \\
\label{ovec2}
& \E(H_{st}^2)= \frac{\sigma^4}{\theta_j^2-\sigma^2}, \ 1\leq s<t\leq K_j, \ \text{in the real symmetric case},\\
\label{ovec3}
& \E( (\Re H_{st})^2)=\E( (\Im H_{st})^2)= \frac{\sigma^4}{2\*(\theta_j^2-\sigma^2)}, \ 1\leq s<t\leq K_j, \ \text{in the Hermitian case},
\end{align}
where
\begin{equation}
\label{chetkum}
\kappa_4(\mu):=\int |x|^4 \*d\mu(x) - (4-\beta)\* (\int |x|^2 \*d\mu(x))^2,
\end{equation}
is the fourth cumulant of $\mu,$ and 

(iii) $U_j$ is a $K_j\times k_j$ such that the ($K_j$-dimensional) columns of $U_j$ are written from the first $K_j$ coordinates of the orthonormal 
eigenvectors
corresponding to $\theta_j.$
\end{theorem}
In \cite{CDF1}, Theorem \ref{thm:caseA} was proved for symmetric marginal distribution satisfying the 
Poincar\'e inequality (\ref{poin})
under an additional technical assumption that 
$k=o(\sqrt{N}), \ $ where $k$ is defined in the paragraph above (\ref{ctheta}).  

Using Theorems 4.1 and 4.2 from \cite{ORS}, one can extend the results of Theorem \ref{thm:caseA} to the case when the entries of $W_N$ are not 
identically distributed provided the distribution of the entries $(W_N)_{il}, \ 1\leq i,l\leq K_j$ does not depend on $N.$

\begin{theorem} 
\label{thm:caseA1}
Let $X_N=\frac{1}{\sqrt{N}} W_N$ be a random real symmetric (Hermitian) Wigner matrix
defined in (\ref{offdiagreal1}-\ref{diagreal2}) (respectively (\ref{offdiagherm1}-\ref{diagherm1})) such that
the distribution of the entries $(W_N)_{il}, \ 1\leq i,l\leq K_j$ does not depend on $N.$
Let us assume that the limits
\begin{equation}
\label{predel100}
m_4(i):=\lim_{N\to \infty} \frac{1}{N}\*\sum_{l: l\not=i} \E |(W_N)_{il}|^4
\end{equation}
exist for $1\leq i\leq K_j.$

Then in case A, the results of Theorem \ref{thm:caseA} hold with $\kappa_4(\mu)$ in (\ref{ovec1}) replaced by 
\begin{equation}
\label{chetkums}
\kappa_4(s):= m_4(s)-(4-\beta)\*\sigma^2, \ s=1, \ldots, K_j.
\end{equation}
\end{theorem}

The next theorem deals with the Case B. 
\begin{theorem}
\label{thm:caseB}
Let $X_N=\frac{1}{\sqrt{N}} W_N$ be a random real symmetric (Hermitian) Wigner matrix
defined in (\ref{offdiagreal1}-\ref{diagreal2}) (respectively (\ref{offdiagherm1}-\ref{diagherm1})) 
such that the off-diagonal entries
$(W_N)_{ij}, \ 1\leq i<j\leq N,$ are i.i.d. random variables with probability distribution $\mu$ and the diagonal entries 
$(W_N)_{ii}, \ 1\leq i< N,$ are i.i.d. random variables with probability distribution $\mu_1.$
In Case B, the $k_j$-dimensional vector 
\[ \Big( c_{\theta_j}\* \sqrt{N}\*(\lambda_{k_1+\ldots +k_{j-1}+i} -\rho_j), \ i=1, \ldots, k_j \Big) \]
converges in distribution to the distribution of the (ordered) eigenvalues of a $k_j\times k_j$ GOE (GUE) matrix with the variance of the matrix 
entries 
given by $\frac{\theta_j^2\*\sigma^2}{\theta_j^2-\sigma^2}$  provided $k=o(\sqrt{N}).$  
\end{theorem}
\begin{remark}
We recall that $k$ has been defined above as 
the minimal number of canonical basis vectors $e_1, \ldots r_N$ required to span the eigenvectors corresponding to the eigenvalues
$\theta_1, \ldots \theta_{J_{\sigma^+}}.$
\end{remark}
Theorem  \ref{thm:caseB} is an immediate extension of the result of Capitaine, Donati-Martin, and F\'eral from \cite{CDF1} to our setting since their 
arguments apply essentially unchanged as soon as Theorem \ref{thm:outliers} is established.

It should be noted that Benaych-Georges, Guionnet, and Maida consider in \cite{BGM} perturbations of a random Wigner matrix by a 
finite rank random matrix with eigenvectors that
are either independent copies of a random vector $v$ with i.i.d. centered components satisfying the log-Sobolev inequality or are obtained by 
Gram-Schmidt orthonormalization of such independent copies.  
The distribution of the outliers is given in Proposition 5.3. of \cite{BGM}.
Let us denote the distribution of the first component of $v$ by $\nu. $  If the fourth cumulant $\kappa_4(\nu)$ of $\nu$ vanishes, the limiting 
distribution of the outliers is similar to the result of Theorem \ref{thm:caseB}, and given by the distribution of the ordered eigenvalues
of a GOE (GUE) matrix.  If the fourth cumulant does not vanish, one has to add a diagonal matrix with i.i.d. real Gaussian entries to a GOE (GUE) matrix.

One of the most important results of \cite{BGM}, \cite{BGM1} concerns the distribution of the ``sticking'' eigenvalues (i.e. the eigenvalues that 
correspond to $|\theta_j| < \sigma).$  In Theorem 5.3 of \cite{BGM},  Benaych-Georges, Guionnet, and Maida prove that their 
limiting distribution is given by the Tracy-Widom law.

Let us briefly describe a key ingredient of the proofs of Theorems \ref{thm:fluctoutliers}-\ref{thm:caseA1}.
We use the notation
\begin{equation}
\label{rezolventa}
R_N(z):=(z\*I_N-X_N)^{-1}, \ z\in \C \setminus [-2\*\sigma, 2\*\sigma],
\end{equation}
for the resolvent of $X_N.$
Clearly, $R_N(z)$ is well defined for $ z\in \C \setminus  \R.$   Since
the spectral norm of $X_N$ converges to $2\*\sigma$ in probability (see e.g. \cite{BYin}, and Proposition 2.1 in \cite{ORS}), 
$R_N(x)$  is well defined for a fixed $x \in \R \setminus [-2\*\sigma, 2\*\sigma]$ with probability going to one.
Since our results will deal with the limiting distribution of random variables 
$\sqrt{N}\left(\langle u, R_N(x) \* v \rangle - g_{\sigma}(x) \* \langle u, v \rangle \right)$ in the limit $N\to \infty,$ 
this should not lead to ambiguity.

Let us consider  a fixed eigenvalue
$\theta_j $ of $A_N$ such that $\theta_j> \sigma$ and denote by $v^{(1)}, \ldots, v^{(k_j)} $ the orthonormal 
eigenvectors of $A_N$ that correspond to the 
eigenvalue $\theta_j. $  Denote by $\Xi^{(j)}_N $ the $k_j\times k_j $ matrix with the entries
\begin{equation}
\label{thetamatrica}
\Xi^{(j)}_{il}:=\sqrt{N}\* \left(\langle v^{(i)}, R_N(\rho_j) v^{(l)}\rangle - g_{\sigma}(\rho_j) \*\delta_{il}\right)=
\sqrt{N}\* \left(\langle v^{(i)}, R_N(\rho_j) v^{(l)}\rangle - \frac{1}{\theta_j}\*\delta_{il}\right),  
\end{equation}
where we recall that $\rho_j=\theta_j +\frac{\sigma^2}{\theta_j}\ $.
The following proposition plays an important part in our proofs.
\begin{proposition}
\label{Prop1}
Let $y_1 \geq \ldots \geq y_{k_j}$ be the ordered eigenvalues of the matrix $\Xi^{(j)}_N.$
Then
\begin{equation}
\label{uzheuzhe}
\sqrt{N}\*(\lambda_{k_1+\ldots +k_{j-1}+i} -\rho_j)+ \frac{1}{g_{\sigma}'(\rho_j)} \* y_i \to 0,
\ i=1, \ldots, k_j,
\end{equation}
in probability.
\end{proposition}
\begin{remark}
A simple computation gives
\begin{equation}
-\frac{1}{g_{\sigma}'(\rho_j)}= \theta_j^2 - \sigma^2.
\end{equation}
\end{remark}

It should be mentioned that 
the key part of the proof of Proposition \ref{Prop1} is a lemma 
from \cite{BGM} which is stated as Lemma \ref{Lemma12} in Section \ref{sec:outliers}.
Proposition \ref{Prop1} indicates that 
the question of the limiting distribution of the outliers of the spectrum of the deformed Wigner matrix $M_N$ can be reduced to the 
question about the limiting distribution of the entries of (\ref{thetamatrica}).

Let us denote by $ \langle u, v \rangle = \sum_1^N \bar{u_i}\*v_i $ the standard Euclidean scalar product in $\C^N. $  
The next theorem deals with the values of the sesquilinear form $\langle u^{(N)}, f(X_N)  v^{(N)}\rangle $ 
where $f$ is a sufficiently nice test function on $\R$ and $u^{(N)}, \ v^{(N)} \in \C^N$ are nonrandom unit vectors in $\C^N$, i.e.
\begin{align}
& \|u^{(N)}\| =\|v^{(N)}\|=1, \  N\geq 1, \ \text{where}\nonumber\\
& \|u \|^2=\langle u, u \rangle= \sum_1^N |u_i|^2. \nonumber 
\end{align}
Without additional assumptions on $u^{(N)}$ and $v^{(N)},$ the sequence 
$$\sqrt{N}\*\left(\langle u^{(N)}, f(X_N) v^{(N)} \rangle- \E \langle u^{(N)}, f(X_N) v^{(N)} \rangle \right)$$ 
does not necessarily converge in distribution.
However, one can show that it is tight. 

We say that a function $f:I\subset \R\to  \R$ belongs to $C^n(I)$ if $f$ and its first $n$ derivatives are continuous on $I.$  Define
\begin{equation*}
\|f\|_{C^n(I)}:=\max\left( |\frac{d^lf}{dx^l}(x)|, \ x\in I, \ 0\leq l\leq n \right).
\end{equation*}
We use the notation 
$C^n_c(\R)$ for the space of $n$ times continuously differentiable functions on $\R$ with compact support.
Define
\begin{align}
\label{normsob}
& \|f\|_{n,1}:= \max( \int_{-\infty}^{\infty} |d^kf/dx^k(x)| \* dx, \ 0\leq k \leq n), \\
\label{normasobolev}
& \|f\|_{n,1,+}:=\max \left( \int_{\R} (|x|+1) \*|\frac{d^lf}{dx^l}(x)| \* dx, \ 0\leq l \leq n \right).
\end{align}

We recall that a function $f:\R \to  \R$ is called Lipschitz continuous on an interval $I \subset \R $ if there exists a constant $C$ such that
\begin{equation}
\label{funkciyaLipschitz}
|f(x)-f(y)| \leq C\*|x-y|, \ \ \text{ for all} \ x,y \in I.
\end{equation}
We define
\begin{equation}
\label{snova}
|f|_{\mathcal{L}, \R}=\sup_{x\not=y}\*\frac{|f(x)-f(y)|}{|x-y|},
\end{equation}
and
\begin{equation}
\label{snova1}
|f|_{\mathcal{L}, \delta}=\sup_{x\not=y, \ x,y \in [-2\sigma -\delta, 2\sigma +\delta]}\*\frac{|f(x)-f(y)|}{|x-y|}.
\end{equation}
\begin{theorem} 
\label{thm:quadratic}
Let $X_N=\frac{1}{\sqrt{N}} W_N$ be a random  real symmetric (Hermitian) Wigner matrix defined in 
(\ref{offdiagreal1}-\ref{diagreal2}) (respectively (\ref{offdiagherm1}-\ref{diagherm1})).
Then the following statements hold:

(i) If $f: \R \to \R$ is a $C^5(\R)$ function such that  $\|f\|_{5,1} $ is finite, 
and $u^{(N)}, \ v^{(N)} \in \C^N, \ N\geq 1,$ are two nonrandom sequences of unit vectors 
(in standard Euclidean norm),
then there exists a constant $Const(\sigma^2, m_5, c_3)$ such that
\begin{equation}
\label{ska1}
\V \left(\langle u^{(N)}, f(X_N) \* v^{(N)} \rangle \right) 
\leq Const \* \frac{\|f\|_{5,1}}{N}.
\end{equation}
In particular, the sequence $ \sqrt{N}\* \left(\langle u^{(N)}, f(X_N) \* v^{(N)} \rangle- \E \langle u^{(N)}, f(X_N) \* v^{(N)} \rangle \right)$ 
is bounded in probability.

(ii) If $f \in C^8_c(\R), $ with $supp(f) \subset [-L, +L],$ where $L$ is some positive number then
there exists a constant $Const(L, \sigma^2, m_5, c_3)$ such that
\begin{align}
\label{ska2}
& \big|\E \langle u^{(N)}, f(X_N) \* v^{(N)}\rangle - \langle u^{(N)}, v^{(N)} \rangle \* \int_{-2\sigma}^{2\sigma} f(x) \* d \mu_{sc}(dx)\big| \\
& \leq Const(L, \sigma^2, m_5, c_3)\* \|f\|_{C^8([-L, +L])} \*\frac{1}{\sqrt{N}}. \nonumber
\end{align}
If, in addition, $f \in C^9(\R)$ and $\|f\|_{9,1,+} $ is finite, 
then
\begin{equation}
\label{ska2p}
\big|\E \langle u^{(N)}, f(X_N) \* v^{(N)}\rangle - \langle u^{(N)}, v^{(N)} \rangle \* \int_{-2\sigma}^{2\sigma} f(x) \* d \mu_{sc}(dx)\big|
\leq Const(\sigma^2, m_5, c_3)\* \|f\|_{9,1,+}\*\frac{1}{\sqrt{N}},
\end{equation}
where $Const(\sigma^2, m_5, c_3)$ depends on $\sigma^2, m_5,$ and $c_3.$

(iii) If the marginal distributions of the entries of $W_N$ satisfy the
Poincar\'e inequality (\ref{poin}) with a uniform constant $\upsilon>0$, and 
$f$ is a Lipschitz continuous function on $[-2\*\sigma -\delta, 2\*\sigma +\delta] $ that
satisfies a subexponential growth condition
\begin{equation}
\label{exprost}
|f(x)|\leq a \* \exp(b\*|x|) \ \ \text{ for all \ } x \in \R,
\end{equation}
for some positive constants $a$ and $b,$
then
\begin{align}
\label{ska3}
&  \P \left( |\langle u^{(N)}, f(X_N) \* v^{(N)}\rangle  - \E \langle u^{(N)}, f(X_N) \* v^{(N)} \rangle| \geq t \right) \\
& \leq  2\*K \* \exp\left( -\frac{\sqrt{\upsilon \*N} \* t}{2 \*|f|_{\mathcal{L},\delta}} \right) + (2\*K+ o(1))\* 
\exp\left(-\frac{\sqrt{\upsilon\*N}}{2}\*\delta\right), \nonumber
\end{align}
where $|f|_{\mathcal{L},\delta} $ is defined in (\ref{snova1}),
\begin{equation}
\label{KK}
K=-\sum_{i\geq 0} 2^i\*\log(1-2^{-1}\*4^{-i}),
\end{equation}
and $\upsilon $ is the constant in the Poincar\'e inequality (\ref{poin}).

(iv) If the marginal distributions of the entries of $W_N$ satisfy the
Poincar\'e inequality (\ref{poin}) with a uniform constant $\upsilon>0$,
and $f$ is a Lipschitz continuous function on $\R,$
then 
\begin{align}
\label{ska4}
&  \P \left( |\langle u^{(N)}, f(X_N) \* v^{(N)}\rangle  - \E \langle u^{(N)}, f(X_N) \* v^{(N)} \rangle| \geq t \right) \\
& \leq  2\*K \* \exp\left( -\frac{\sqrt{\upsilon \*N} \* t}{2\*|f|_{\mathcal{L}, \R}} \right), \nonumber
\end{align}
where $|f|_{\mathcal{L},\R} $ is defined in (\ref{snova}).

(v)  If the marginal distributions of the entries of $W_N$ satisfy the
Poincar\'e inequality (\ref{poin}) with a uniform constant $\upsilon>0$,
$ f\in C^8(\R),$ and  $f$ satisfies the subexponential growth condition
(\ref{exprost}), then 
\begin{equation}
\label{ska5}
\E \langle u^{(N)}, f(X_N) \* v^{(N)}\rangle = \langle u^{(N)}, v^{(N)} \rangle \* \int_{-2\sigma}^{2\sigma} f(x) \* d \mu_{sc}(dx) 
+ O\left(\frac{1}{\sqrt{N}}\right).
\end{equation}
\end{theorem}

We finish this section by formulating our last theorem, Theorem \ref{thm:Local}, which allows us to extend Theorem \ref{thm:caseA} (see Remark \ref{rem}
in Section \ref{sec:loc}).
Assume that that the off-diagonal entries
$(W_N)_{ij}, \ 1\leq i<j\leq N,$ are i.i.d. random variables with probability distribution $\mu$ and the diagonal entries 
$(W_N)_{ii}, \ 1\leq i< N,$ are i.i.d. random variables with probability distribution $\mu_1.$

Let us consider $ u^{(N)}, \ v^{(N)} \in \C^N \ $ that are independent of $N$ for all $N \geq N_0,\  $
in a sense that only a fixed finite number of the coordinates of $u^{(N)}, \ v^{(N)} \ $ are non-zero and the coordinates do not change with $N$ for
$N\geq N_0.$
In this case, we can write $u^{(N)}=u, \ v^{(N)}=v, \ $ with the understanding that as the dimension $N$ grows, one just adds more zero coordinates to 
$u$ and $ v. \ $  As an immediate consequence of the results of Theorem 1.1 (real symmetric case) and Theorem 1.5 (Hermitian case) 
in \cite{PRS}, the random sequence 
\begin{equation}
\label{skalyar}
\sqrt{N}\*\left(\langle u, R_N(z) v \rangle- g_\sigma(z)\* \langle u, v \rangle \right)
\end{equation}
converges in distribution as $N \to \infty.$
Without loss of generality, we will consider the real symmetric case; the Hermitian case is essentially identical.
Let $m $ be an arbitrary fixed positive integer. Denote by $R^{(m)}(z) $ the $m\times m$ upper-left corner of the matrix $R_N(z).$
Theorem 1.1 in \cite{PRS} states that a matrix-valued random field 
\begin{equation}
\label{upsNN}
\Upsilon_N(z)=\sqrt{N} \* \left(R^{(m)}(z)-g_{\sigma}(z)\*I_m \right), \ z \in \C \setminus [-2\*\sigma, 2\*\sigma],
\end{equation}
with values in the space of complex symmetric $m\times m$ matrices,
converges in finite-dimensional distributions to a random field 
\begin{equation}
\label{upsaa}
\Upsilon (z) = g_{\sigma}^2(z)\*(W^{(m)} + Y(z)),
\end{equation}
where $W^{(m)}$ is the $m\times m$ upper-left corner submatrix of a Wigner matrix $W_N, \ g_{\sigma}(z) $ is the Stieltjes transform (\ref{steltsem}) 
of the Wigner semicircle law,  and
$$Y(z)=\left(Y_{ij}(z)\right), Y_{ij}(z)=Y_{ji}(z), \ 1\leq i,j \leq m, $$
is a Gaussian random field with the covariance matrix given by 
the formulas (1.18)-(1.23) in the real-symmetric case and (1.50)-(1.55) in the Hermitian case in \cite{PRS}.  It is important to 
note that $Y_{ij}(z), \ 1\leq i\leq j \leq m, $ 
are independent random processes for different indices $(ij).$

Let us extend the definition of $\Upsilon(z) $ 
to that of an infinite-dimensional  matrix $ \Upsilon(z)_{pq}, \ 1\leq p,q <\infty, \ $ using
the formulas (1.18)-(1.23) (respectively (1.50)-(1.55)) from \cite{PRS}. Thus, the r.h.s. in (\ref{upsaa}) defines now the 
$m\times m$ upper-left corner 
of the infinite matrix $\Upsilon(z).$ 
Then Theorem 1.1 of \cite{PRS} implies that
\begin{equation}
\label{shodnya200}
\sqrt{N}\*\left(\langle u, R_N(z) v \rangle- g_\sigma(z)\* \langle u, v \rangle \right)\to
\langle u, \Upsilon(z)\*v \rangle,
\end{equation}
in distribution.

Let $u, v \in l^2(\N).\ $ It follows from the Kolmogorov three-series theorem (see e.g. 
\cite{Dur}) that $ \ \langle u, \Upsilon(\theta_j)\* v \rangle $ is well defined as an infinite sum 
of centered random variables with summable variances.
For our analysis of the outliers in the spectrum of finite-rank deformations of Wigner matrices, it will be useful to have the following result.  
\begin{theorem}
\label{thm:Local}
Let $X_N=\frac{1}{\sqrt{N}} W_N$ be a random real symmetric (Hermitian) Wigner matrix 
defined in (\ref{offdiagreal1}-\ref{diagreal2}) (respectively (\ref{offdiagherm1}-\ref{diagherm1}))
such that that the off-diagonal entries
$(W_N)_{ij}, \ 1\leq i<j\leq N,$ are i.i.d. random variables with probability distribution $\mu$ and the diagonal entries 
$(W_N)_{ii}, \ 1\leq i< N,$ are i.i.d. random variables with probability distribution $\mu_1.$

Let $l$ be a fixed positive integer, $u_1, \ldots, u_l, \ $ be a collection of non-random vectors in $ l^2(\N), \ $ and let
$u^{(N)}_p, \ 1\leq p \leq l, \ $ denote the projection of $u_p$ 
to the subspace spanned by the first $N$ standard basis vectors $e_1, \ldots, e_N. $
Then the joint distribution of 
$$ \sqrt{N}\*\left(\langle u^{(N)}_p, R_N(z) u^{(N)}_q \rangle- g_\sigma(z)\* \langle u^{(N)}_p, u^{(N)}_q \rangle \right), \ \ 1\leq p,q\leq l, $$
converges weakly to the joint distribution of
$ \ \langle u_p, \Upsilon(z) u_q \rangle, \ \ 1\leq p,q\leq l.$
\end{theorem}
The rest of the paper is organized as follows.  Section \ref{sec:mathvar} is devoted to the estimates on the mathematical expectation and the variance
of the values of the resolvent sesquilinear form $\langle u^{(N)}, R_N(z) v^{(N)} \rangle, $ where $u^{(N)}, \ v^{(N)} $ are arbitrary non-random unit 
vectors in $\C^N.$
Using the estimates obtained in Section \ref{sec:mathvar}, 
we prove Theorems \ref{thm:quadratic} in Section \ref{sec:proofquadratic}.
Theorem \ref{thm:fluctoutliers} is proved in Section \ref{sec:outliers}. Finally, Theorems  \ref{thm:caseA}, \ref{thm:caseA1}, and \ref{thm:Local} are 
proved in Section \ref{sec:loc}. In the Appendix, we discuss tools used throughout the paper. 

We would like to thank A. Guionnet for bringing our attention to the preprints \cite{BGM} and \cite{BGM1}.

\section{ \bf{Mathematical Expectation and Variance of Resolvent Sesquilinear Form}}
\label{sec:mathvar}
This section is devoted to the proof of the main building block Theorem \ref{thm:quadratic}, namely Proposition \ref{proposition:prop1}.

Without loss of generality, we can restrict our attention to the real symmetric case. 
Let $u^{(N)}=(u_1, \ldots, u_N), \ 
v^{(N)}=(v_1, \ldots, v_N) $  be nonrandom unit vectors in $\C^N.$ When it does not lead to ambiguity, we will omit the superscript in 
$u^{(N)}$ and $v^{(N)}.$
Define
\begin{equation}
\label{etann}
\eta_N:= \langle u^{(N)}, R_N(z) v^{(N)} \rangle= \sum_{ij} \bar{u_i} \* R_{ij}\*v_j. 
\end{equation}
When it does not lead to ambiguity we will use the shorthand notation, $R_{ij}$,  for the $ij$-th entry $(R_N(z))_{ij},$ of the resolvent matrix 
$R_N(z).$
\begin{proposition}
\label{proposition:prop1}
Let $X_N=\frac{1}{\sqrt{N}} W_N$ be a random real symmetric (Hermitian) Wigner matrix defined in (\ref{offdiagreal1}-\ref{diagreal2})
((\ref{offdiagherm1}-\ref{diagherm1})), $R_N(z)=(z\*I_N-X_N)^{-1}, \ z \in \C\setminus \R, $ and  $u^{(N)}=(u_1, \ldots, u_N), \ 
v^{(N)}=(v_1, \ldots, v_N) $  be nonrandom unit vectors in $\C^N.$
Then
\begin{equation}
\label{peredacha}
\E \eta_N=\E \langle u^{(N)}, R_N(z) v^{(N)} \rangle = g_{\sigma}(z) \* \langle u^{(N)}, v^{(N)} \rangle 
+O\left(\frac{1}{|\Im z|^7\*\sqrt{N}}\right), 
\end{equation}
uniformly on bounded subsets of $\C\setminus \R, $
\begin{align}
\label{peredachauh}
& \E \langle u^{(N)}, R_N(z) v^{(N)} \rangle =g_{\sigma}(z) \* \langle u^{(N)}, v^{(N)} \rangle +
O\left((|z|+M)\*\frac{P_8(|\Im z|^{-1})}{\sqrt{N}}\right),\\
\label{peredachavar}
& \V \eta_N=\V \langle u^{(N)}, R_N(z) v^{(N)} \rangle = 
O\left(\frac{P_8(|\Im z|^{-1})}{N}\right),
\end{align}
uniformly on $\C\setminus \R,$ where 
$P_l(x), \ l\geq 1,$ denotes a polynomial of degree $l$ with fixed positive coefficients, and
$M$ is some constant.
\end{proposition}
\begin{remark}
In the case when $u^{(N)} $ and $ v^{(N)}$ are standard basis vectors, 
$u=e_i, \ v=e_j, $ the mathematical expectation and the variance of $\langle u^{(N)}, R_N(z) v^{(N)} \rangle$ have been studied in \cite{PRS}.
In particular, it has been shown there in Proposition 2.1 and (3.27) that
\begin{equation}
\label{peredacha1}
\E R_{ii}= g_{\sigma}(z) + O\left(\frac{1}{|\Im z|^6\*N}\right),
\end{equation}
uniformly on bounded subsets of $\C \setminus \R,$ and
\begin{align}
\label{peredacha1p}
& \E R_{ii}=g_{\sigma}(z)+ O\left((|z|+M)\*\frac{P_7(|\Im z|^{-1})}{N}\right), \\
\label{peredacha2}
& \E R_{ij}= O\left(\frac{P_5(|\Im z|^{-1})}{N}\right), \ \E R_{ij}=O \left( \frac{P_9(|\Im z|^{-1})}{N^{3/2}}\right), \ i\not=j, \\
\label{peredacha3}
& \V R_{ij} = O \left( \frac{P_6(|\Im z|^{-1})}{N}\right), \  1\leq i, j\leq N, 
\end{align}
uniformly on $\C \setminus \R. $
\end{remark}
\begin{remark}
In \cite{EYY}, Erd\"os, Yau, and Yin studied generalized Wigner matrices (defined at the beginning of Section 2 of  \cite{EYY}), and obtained 
the following estimates provided the marginal distributions have subexponential tails
\begin{align}
\label{eyy1}
& \P  \left \{ max_i |R_{ii}(z) - g_{\sigma}(z)| \geq \frac{(\log N)^l}{(N \* |\Im z |)^{1/3}} \right \} 
\leq C \exp \left[ -c (\log N)^{\phi \*l}\right], \\
\label{eyy2}
& \P  \left \{ max_{i\not=j} |R_{ij}(z)| \geq \frac{(\log N)^l}{(N \* |\Im z |)^{1/2}} \right \} \leq C \exp \left[ -c (\log N)^{\phi \*l}\right],
\end{align}
where $0<\phi<1, \ C\geq 1, c>0$ are some constants, $4/\phi\leq l \leq C \log N/ \log \log N,$ \\
$ N^{-1}\* (\log N)^{10\*l} <\Im z \leq 10, \ 
|\Re z| \leq 5\sigma, $ and $N$ is sufficiently large.
\end{remark}
\begin{remark}
It follows from our proofs that the error term on the r.h.s. of (\ref{peredacha}) can be replaced by 
$ O\left(\frac{ \min ( \|u\|_{1}, \|v\|_{1})}{|\Im z|^7\*N}\right), \ $
where $\|u\|_{1}=\sum_{i=1}^N |u_i|. $
\end{remark}
The rest of the section is devoted to the proof of Proposition \ref{proposition:prop1}.
\begin{proof}
Without loss of generality, we can restrict our attention to the real symmetric case. The proof in the Hermitian case is very similar.
We start by proving (\ref{peredacha}).
Using $(z\*I_N -X_N)\*R_N(z)=I_N,$  we write
\begin{equation}
\label{odin}
z\*\E \sum_{ij} \bar{u_i} \*R_{ij} \*v_j = \E \sum_{ijk} \bar{u_i} \* \left(\delta_{ij} + X_{ik}\*R_{kj} \right) \*v_j=
\langle u, v \rangle +\sum_{ijk} \*\bar{u_i} \*v_j\*\E(X_{ik}\*R_{kj}).
\end{equation}
Applying the decoupling formula (\ref{decouple}) and (\ref{vecher1}-\ref{vecher2}) to the term $\E(X_{ik}\*R_{kj})$ in (\ref{odin}), we obtain
\begin{align}
\label{dva}
& z\* \E \eta_N = \langle u, v \rangle +\sigma^2 \* \E \left( \eta_N\* \tr_N R \right) +
\frac{\sigma^2}{N} \E \left( \langle u, (R_N(z))^2 v \rangle\right) \\
\label{dva1015}
& + \sum_{i,j}  \frac{\V [(W_N)_{ii}]-2\*\sigma^2}{N}\*\bar{u_i}\* v_j\* 
\E (R_{ii}\*R_{ij}) + r_N, 
\end{align}
where $\eta_N $ is defined in (\ref{etann}),
and $r_N$ contains the third and the fourth cumulant terms corresponding to $p=2$ and $p=3$ in the decoupling formula (\ref{decouple})
for $i\not=k,$ 
and the error terms due to the truncation of the decoupling formula (\ref{decouple}) for $i\not=k$ at $p=3$
and for $i=k$ at $p=1.$

It follows from 
\begin{equation*}
|\V [(W_N)_{ii}]-2\*\sigma^2|\leq const(\sigma^2, c_3), 
\end{equation*}
that the first term in (\ref{dva1015}) 
can be written as the mathematical expectation of $\frac{1}{N}\*\langle a, R_N(z) v\rangle,$
where the vector $a$ has coordinates $(\V [(W_N)_{ii}]-2\*\sigma^2)\* \overline{R_{ii}}\*u_i, \ 1\leq i\leq N.$ 
Using (\ref{resbound}), one obtains by estimating $\|a\| $ from above that
\begin{equation}
\label{porto}
\sum_{i,j}  \frac{\V [(W_N)_{ii}]-2\*\sigma^2}{N}\*\bar{u_i}\* v_j\* 
\E (R_{ii}\*R_{ij})=O\left(\frac{1}{N\*|\Im z|^2}\right).
\end{equation}

The third cumulant terms $(p=2)$ give
\begin{align}
& \frac{1}{2!\*N^{3/2}} [ 4\* \E (\sum_{i,j,k:i\not=k} \kappa_3(i,k)\*\bar{u_i}\*R_{ij}\*R_{ik}\*R_{kk}\*v_j) + 
2\*\E (\sum_{i,j,k:i\not=k} \kappa_3(i,k)\*\bar{u_i}\*R_{ii}\*R_{kk}\*R_{kj}\*v_j) \nonumber \\
\label{tretii}
&+ 2\*\E (\sum_{i,j,k:i\not=k} \kappa_3(i,k)\*\bar{u_i}\*R_{ki}\*R_{ki}\*R_{kj}\*v_j) ],
\end{align}
where by $\kappa_3(i,k)$ we denote the third cumulant of $(W_N)_{ik}.$  We note that
\begin{equation*}
|\kappa_3(i,k)|\leq Const(m_5),
\end{equation*}
uniformly in $i\not=k$ and $N.$
To estimate the absolute value of the first term in (\ref{tretii}), we first sum with respect to $j$ and then 
use the Cauchy-Schwarz inequality and (\ref{resbound}) to obtain
\begin{align}
& |\E \sum_{i,j,k:i\not=k} \kappa_3(i,k)\* \bar{u_i}\*R_{ij}\*R_{ik}\*R_{kk}\*v_j| = |\E \sum_{i\not=k} \kappa_3(i,k)\*
\bar{u_i}\*R_{ik}\*R_{kk}\*(R\*v)_i| \nonumber \\
\label{tretii1}
& \leq Const(m_5)\*\E \left( \sqrt{\sum_k|R_{kk}|^2} \* \sum_{i=1}^N |u_i|\*\|Re_i\| \*|(R\*v)_i| \right) \leq  
\sqrt{N} \*\frac{Const(m_5)}{|\Im z|^3}.
\end{align}
To estimate the absolute value of the second term in (\ref{tretii}), we write
\begin{align}
& |\E \sum_{i,j,k:i\not=k} \kappa_3(i,k)\* \bar{u_i}\*R_{ii}\*R_{kk}\*R_{kj}\*v_j| = 
|\E \sum_{i\not=k} \kappa_3(i,k)\*\bar{u_i}\*R_{ii}\*R_{kk}\*(R\*v)_k| \nonumber \\
\label{tretii2}
& \leq Const(m_5)\* \E \left( \sum_{i,k} |u_i|\* \|R\|^2 \*|(R\*v)_k| \right) \leq 
Const(m_5)\* \sqrt{N}\* \sum_{i} |u_i| \* \|v\| \*\E \|R_N(z)\|^3 \leq  N \*
\frac{Const(m_5)}{|\Im z|^3}.
\end{align}
Finally, we bound the last of the third cumulant terms in (\ref{tretii}) as
\begin{align}
& |\E \sum_{i,j,k:i\not=k} \kappa_3(i,k)\*\bar{u_i}\*R_{ki}\*R_{ki}\*R_{kj}\*v_j| = 
|\E \sum_{i\not=k} \kappa_3(i,k)\*\bar{u_i}\*(R_{ki})^2\*(R \*v)_k| \nonumber \\
\label{tretii3}
& \leq Const(m_5)\*\E \sum_{ik} |u_i|\*|R_{ki}|^2\*\|R_N(z) \| \*\|v\| \leq  \sqrt{N} \*\frac{Const(m_5)}{|\Im z|^3},
\end{align}
where we again used (\ref{resbound}) and 
\begin{equation*}
\sum_k |R_{ki}|^2=\|R_N(z)\* e_i\|^2 \leq \|R_N(z)\|^2 \leq \frac{1}{|\Im z|^2}. 
\end{equation*}
Combining the bounds (\ref{tretii1}-\ref{tretii3}),  we see that the contribution of the 
third cumulant terms to $r_N$ in (\ref{dva}-\ref{dva1015}) is
bounded from above by 
$ O\left(\frac{1}{|\Im z|^3\*\sqrt{N}}\right).$
The fourth cumulant terms give
\begin{align}
\label{chetvert}
& \frac{1}{3!\*N^2} \* [ 18\* \E (\sum_{i,j,k:i\not=k} \kappa_4(i,k)\*\bar{u_i}\*R_{ii}\*R_{ik}\*R_{kk}\*R_{kj}\*v_j) +  
6 \*\E (\sum_{i,j,k:i\not=k} \kappa_4(i,k)\* \bar{u_i}\*R_{ii}\*(R_{kk})^2\*R_{ij}\*v_j) \\
& + 18\*\E (\sum_{i,j,k:i\not=k} \kappa_4(i,k)\* \bar{u_i}\*(R_{ki})^2\*R_{kk}\*R_{ij}\*v_j)  + 
6 \* \E (\sum_{i,j,k:i\not=k} \kappa_4(i,k)\* \bar{u_i}\*(R_{ki})^3\*R_{kj}\*v_j) ]. \nonumber
\end{align}
To estimate the absolute value of the first term in (\ref{chetvert}), we note that
\begin{align}
\label{chetvert1}
& |\E \sum_{i,j,k:i\not=k}  \kappa_4(i,k)\* \bar{u_i}\*R_{ii}\*R_{ik}\*R_{kk}\*R_{kj}\*v_j| =  
|\E \sum_{i\not=k}  \kappa_4(i,k)\*\bar{u_i}\*R_{ii}\*R_{ik}\*R_{kk}\* (R \*v)_k|\\
& \leq Const(m_5)\*\E \left( \sum_{ik} |\bar{u_i}| \* \|R\|^2 \* |R_{ik}| \* |(R \*v)_k| \right)  
\leq \sqrt{N} \*\frac{Const(m_5)}{|\Im z|^4},
\nonumber
\end{align}
where we used the bound $$ \sum_k |R_{ik}|\*|(R \*v)_k| \leq \|R_N(z)\* e_i\|\*\|R_N(z)\*v\|\leq \|R_N(z)\|^2 \*\|v\|, $$  (\ref{resbound}),
and the fact that the fourth cumulants of $(W_N)_{ik}$ are uniformly bounded in absolute value by some constant $Const(m_5).$

To estimate the second term in (\ref{chetvert}), we write
\begin{align}
\label{chetvert2}
& |\E \sum_{i,j,k:i\not=k} \kappa_4(i,k)\*\bar{u_i}\*R_{ii}\*(R_{kk})^2\*R_{ij}\*v_j|=
|\E \sum_{i\not=k} \kappa_4(i,k)\*\bar{u_i}\*R_{ii}\*(R_{kk})^2\*(R \*v)_i| \leq \\
& Const(m_5)\*\E (\sum_k |R_{kk}|^2 \* \|R_N(z)\| \* \sum_i |u_i|\* |(R \*v)_i|) \leq N \*\frac{Const(m_5)}{|\Im z|^4}.
\nonumber
\end{align}
The other two terms in (\ref{chetvert}) are estimated in a similar fashion.  Each of them is 
$O\left(\frac{N \* \|u\| \* \|v\|}{|\Im z|^2}\right).$  Therefore, the fourth cumulant terms give the contribution 
$O\left(\frac{1}{N\*|\Im z|^4}\right)$  to $r_N$ in (\ref{dva}-\ref{dva1015}). 

Finally, we estimate the error terms due to the truncation of the decoupling formula at $p=3$ for $i\not=k$ and at $p=1$ for $i=k.$
Here, we treat the error term due to the truncation of the decoupling formula at $p=3$ for $i\not=k.$   The second error term can be treated in a similar 
way. To estimate the error term, we have to consider expressions of the following form
\begin{equation}
\label{trun}
N^{-5/2} \E  \left (  \sum_{ik} |\kappa_5(i,k)\* \sup |u_i| |R^{(1)}_{ab}|\*|R^{(2)}_{cd}|\*|R^{(3)}_{ef}|\*|R^{(4)}_{pq}|\*|(R^{(5)}v)_s| \right),
\end{equation}
where $a,b,c,d,e,f,p,q,s \in \{i,k\}, \ $ the supremum in (\ref{trun}) is considered over the resolvents 
$R^{(l)}= (z-X_N^{(l)})^{-1}, \ l=1,\ldots 5 $ of 
rank two perturbations $X_N^{(l)}=X_N +x\*E_{ik} $ of $X_N$ with  $(E_{ik})_{jh}=\delta_{ij}\*\delta_{kh} +
\delta_{ih}\*\delta_{kj}. $  Estimating each entry of $R^{(l)}$ by $\frac{1}{|\Im z|}, $ taking into account that
\begin{equation*}
\sum_{i=1}^N |u_i|\leq \sqrt{N}\*\|u\|=\sqrt{N},
\end{equation*}
and using the fact that the fifth cumulants of the off-diagonal entries of $W_N$ are uniformly bounded, 
we bound (\ref{trun}) from above by  $O\left(\frac{1}{N\*|\Im z|^5}\right).\ $   

Combining the estimates of the third and the fourth cumulant terms and the truncation error term, we can rewrite the Master equation (\ref{dva}) as
\begin{equation}
\label{dvaa}
z\* \E \eta_N = \langle u, v \rangle +\sigma^2 \* \E \left( \eta_N\* \tr_N R \right) +
\frac{\sigma^2}{N} \E \left( \langle u, (R_N(z))^2 v \rangle\right) + O\left(\frac{P_5(|\Im z|^{-1})}{\sqrt{N}}\right),
\end{equation}
where we recall that by $P_l$ we denote a polynomial of degree $l$ with positive coefficients that do not depend on $N.$

Since 
\begin{equation*}
|\langle u, (R_N(z))^2 v \rangle| \leq \|u\|\*\|v\|\*\frac{1}{|\Im z|^2},
\end{equation*}
we obtain
\begin{equation}
\label{dvaaa}
z\* \E \eta_N = \langle u, v \rangle +\sigma^2 \* \E \left( \eta_N\* \tr_N R \right) + O\left(\frac{P_5(|\Im z|^{-1})}{\sqrt{N}}\right).
\end{equation}
Finally, we have to estimate the term $\E \left( \eta_N\* \tr_N R \right)$ in the Master equation.  We write
\begin{align}
\label{riorio}
& | \E \left(  \tr_N R \* \langle u, R_N(z) v \rangle \right) - g_{\sigma}(z) \* \E \langle u, R_N(z) v \rangle| \leq
\left(\V\left( \langle u, R_N(z) v \rangle \right) \right)^{1/2} \* \left(\V( \tr_N R) \right)^{1/2}  \\
\label{rioriorio}
& + | g_N(z)-g_{\sigma}(z)| \* \|u\| \* \|v\| \* \frac{1}{|\Im z|}, 
\end{align}
where we use the notation
\begin{equation}
\label{gn}
g_N(z):=\E \tr_N R_N(z).
\end{equation}
The variance $\V (\tr_N R_N(z))$ has been estimated in Proposition 2 of \cite{Sh} as
\begin{equation}
\label{shcherb}
\V (\tr_N R_N(z))= O\left(\frac{1}{|\Im z|^{4}\*N^2}\right),
\end{equation}
uniformly on $\C\setminus\R.$  It follows from the proof of (\ref{shcherb}) that the bound is valid provided the fourth moments of the off-diagonal 
entries are uniformly bounded and the second moments of the diagonal entries are uniformly bounded (\cite{Shch}).
Applying the bound $|\langle u, R_N(z) v \rangle| \leq \frac{\|u\|\*\|v\|}{|\Im z|}= \frac{1}{|\Im z|} \ $  and (\ref{peredacha1}), we obtain
\begin{equation*}
\E \left(  \tr_N R \* \langle u, R_N(z) v \rangle \right)= g_{\sigma}(z) \* \E \langle u, R_N(z) v \rangle +  
O\left(\frac{P_7(|\Im z|^{-1})}{N}\right),
\end{equation*}
uniformly on bounded subsets of $\C\setminus\R.$
This allows us to write
the Master Equation for $\eta_N= \langle u^{(N)}, R_N(z) v^{(N)} \rangle $ as
\begin{equation}
\label{dvasto}
z\* \E \eta_N=\langle u, v \rangle + \sigma^2\* g_{\sigma}(z) \* \E \eta_N + O\left( \frac{P_7(|\Im z|^{-1})}{\sqrt{N}}\right),
\end{equation}
uniformly on bounded subsets of $\C\setminus\R.$
Since $\ z-\sigma^2\* g_{\sigma}(z)=1/g_{\sigma}(z) \ $ and $g_{\sigma}(z) $ is bounded,
we arrive at 
\begin{equation}
\label{MGU}
\E\langle u, R_N(z) v \rangle= g_{\sigma}(z)\* \langle u, v \rangle +  O\left( \frac{P_7(|\Im z|^{-1})}{\sqrt{N}}\right).
\end{equation}
which is exactly the estimate (\ref{peredacha}) of Proposition \ref{proposition:prop1}.

To prove (\ref{peredachauh}), we note that (\ref{riorio}-\ref{rioriorio}), (\ref{shcherb}) and (\ref{peredacha1p}) imply
\begin{equation*}
\E \left(  \tr_N R \* \langle u, R_N(z) v \rangle \right)= g_{\sigma}(z) \* \E \langle u, R_N(z) v \rangle +  
O\left((|z|+M)\*\frac{P_8(|\Im z|^{-1})}{N}\right),
\end{equation*}
uniformly on $\C\setminus\R.$  Therefore, one can rewrite (\ref{dvaaa}) as
\begin{equation}
\label{dvasto111}
z\* \E \eta_N=\langle u, v \rangle + \sigma^2\* g_{\sigma}(z) \* \E \eta_N + O\left( (|z|+M)\*\frac{P_8(|\Im z|^{-1})}{\sqrt{N}}\right),
\end{equation}
uniformly on $\C\setminus\R,$
which implies (\ref{peredachauh}).

Now, we turn our attention to the proof of (\ref{peredachavar}). The key part of the proof is the following lemma.
\begin{lemma}
\label{LemmaVar}
Let 
$X_N=\frac{1}{\sqrt{N}} W_N$ be a random real symmetric (Hermitian) Wigner matrix defined in (\ref{offdiagreal1}-\ref{diagreal2})
((\ref{offdiagherm1}-\ref{diagherm1})), $R_N(z)=(z\*I_N-X_N)^{-1}, \ z \in \C\setminus\R, $ and  $u^{(N)}=(u_1, \ldots, u_N), \ 
v^{(N)}=(v_1, \ldots, v_N) $  be nonrandom unit vectors in $\C^N.$
Then
\begin{align}
\label{101}
& (z - \sigma^2 \*g_N(z)) \V( \langle u^{(N)},R_N(z) v^{(N)} \rangle ) = 
\sqrt{ ( \V( \langle u^{(N)},R_N(z)\* v^{(N)} \rangle )} O\left( \frac{P_3(|\Im z|^{-1})}{\sqrt{N}} \right) + \\
& O\left( \frac{P_6(|\Im z|^{-1})}{N} \right), \nonumber
\end{align}
uniformly in $z\in \C\setminus \R,$ where $g_N(z)$ is defined in (\ref{gn}).
\end{lemma}
\begin{proof}
As always, we will suppress the dependence on $N$ in  $u=u^{(N)} $ and $v=v^{(N)},$
and use the notation $\eta_N=\langle u,R_N(z) v \rangle.$
Clearly, $\V(\eta_N)= \E |\eta_N|^2- |\E \eta_N|^2, \ $ and
$ \overline{\langle u,R_N(z) v \rangle}= \langle v,R_N(\bar{z}) u \rangle.\ $
We start with the following form of the Master equation for $\eta_N,$
\begin{align}
\label{two}
& z\* \E \eta_N = \langle u, v \rangle +\sigma^2 \* g_N(z) \* \E \eta_N \\
\label{twoa}
& + \frac{1}{N^{3/2}} \* \E (\sum_{i,j,k:i\not=k} \kappa_3(i,k)\*\bar{u_i}\*R_{ii}(z)\*R_{kk}(z)\*R_{kj}(z)\*v_j) +
O\left(\frac{P_5(|\Im z|^{-1})}{N}\right),
\end{align}
uniformly on $\C\setminus \R.$ 
We singled out in (\ref{twoa}) the only term in $r_N$ which is $O(N^{-1/2}),$ namely
(\ref{tretii2}).  As we have shown above, all other terms in $r_N$ can be estimated as $O\left(\frac{P_5(|\Im z|^{-1})}{N}\right).$
Multiplying both sides of the equation by $\overline{\E \eta_N},$ we obtain
\begin{align}
\label{twotwo}
& z\* |\E \eta_N|^2= \langle u, v \rangle \* \overline{\E \eta_N} +
\sigma^2 \* |\E \eta_N|^2\* g_N(z) \\
& + \frac{1}{N^{3/2}} \* \E \big(\sum_{i,j,k:i\not=k}  \kappa_3(i,k)\*\bar{u_i}\*R_{ii}(z)\*R_{kk}(z)\*R_{kj}(z)\*v_j\big)  \*\overline{\E \eta_N}
+ O\left(\frac{P_6(|\Im z|^{-1})}{N}\right). \nonumber
\end{align}
Our next goal is to obtain the Master equation for $ z\* \E (|\eta_N|^2). $
As before, we use the resolvent identity (\ref{resident}) to write
\begin{align}
\label{odin1}
& z\* \E (|\eta_N|^2)
=z\*\E [\sum_{ij} \bar{u_i} \*R_{ij}(z) \*v_j \* \overline{\eta_N}]= \E [\sum_{ijk} \bar{u_i} \* \left(\delta_{ij} + X_{ik}\*R_{kj}(z) \right) \*v_j \* 
\overline{\eta_N}] \\
\label{odin2}
& = \langle u, v \rangle \* \E \overline{\eta_N} +\sum_{ijk} \*\bar{u_i} \*v_j\*\E(X_{ik}\*R_{kj}(z)\*\overline{\eta_N}). 
\end{align}
Applying the decoupling formula (\ref{decouple}) and (\ref{vecher1}-\ref{vecher2}) 
to the term $\E(X_{ik}\*R_{kj}(z)\*\overline{\eta_N})$ in (\ref{odin1}-\ref{odin2}), we obtain
\begin{align}
\label{dva1}
& z\* \E (|\eta_N|^2) = \langle u, v \rangle \* \overline{ \E \eta_N} +\sigma^2 \* \E \left( |\eta_N|^2 \* \tr_N R_N(z) \right) +
\frac{\sigma^2}{N} \* \E \left( \langle u, (R_N(z))^2 v \rangle\*\overline{\eta_N}\right) \\
\label{dva2}
& \sum_{i,j}  \frac{\V [(W_N)_{ii}]-2\*\sigma^2}{N}\*\bar{u_i}\* v_j\* 
\E [R_{ii}\*R_{ij} \* \overline{\eta_N}] + \frac{\sigma^2}{N} \* 
\E \left(\sum_{i,j,k:i\not=k} \bar{u_i}\*R_{kj}(z)\*v_j \* \frac{\partial \langle v,R_N(\bar{z}) u \rangle}{\partial X_{ik}}\right) \\
\label{dva3}
& + \sum_{i,j}  \frac{\V [(W_N)_{ii}]}{N}\*  
\E \left(\bar{u_i}\*R_{ij}(z)\*v_j \* \frac{\partial \langle v,R_N(\bar{z}) u \rangle}{\partial X_{ii}} \right) +r_N, 
\end{align}
where $r_N$ contains the third and the fourth cumulant terms corresponding to $p=2$ and $p=3$ in (\ref{decouple}) for $k=i$,
and the error due to the truncation of the decoupling formula (\ref{decouple}) at $p=3$ for $k\not=i$ and at $p=1$ for $k=i.$
Clearly,
\begin{equation}
\label{amazon}
\frac{\sigma^2}{N} \* \E \left( \langle u, (R_N(z))^2 v \rangle\*\overline{\eta_N}\right) = O\left(\frac{1}{|\Im z|^3\*N}\right).
\end{equation}
For $k\not=i, $ we have
\begin{align}
\label{102}
& \frac{\partial \langle v,R_N(\bar{z}) u \rangle}{\partial X_{ik}}= (R_N(\bar{z})\bar{v})_i\* (R_N(\bar{z})u)_k + 
(R_N(\bar{z})\bar{v})_k\* (R_N(\bar{z})u)_i, \\
\label{103}
& \frac{\partial^2 \langle v,R_N(\bar{z}) u \rangle}{\partial X_{ik}^2}= 2\* R_{ii}(\bar{z})\* (R_N(\bar{z})\bar{v})_k\* (R_N(\bar{z})u)_k \\
\label{104}
& +2\* R_{ik}(\bar{z})\* (R_N(\bar{z})\bar{v})_i\* (R_N(\bar{z})u)_k + 2\* R_{ik}(\bar{z})\* (R_N(\bar{z})\bar{v})_k\* (R_N(\bar{z})u)_i \\
\label{105}
& + 2\* R_{kk}(\bar{z})\* (R_N(\bar{z})\bar{v})_i\* (R_N(\bar{z})u)_i, \\
\label{106}
& \frac{ \partial (R_N(\bar{z})w)_i  }{\partial X_{ik}}=R_{ii}(\bar{z})\* (R_N(\bar{z})w)_k + R_{ik}(\bar{z})\* (R_N(\bar{z})w)_i.
\end{align}
For $k=i, $ we have
\begin{align}
\label{107}
& \frac{\partial \langle v,R_N(\bar{z}) u \rangle}{\partial X_{ii}}= (R_N(\bar{z})\bar{v})_i\* (R_N(\bar{z})u)_i,\\
\label{108}
& \frac{\partial^2 \langle v,R_N(\bar{z}) u \rangle}{\partial X_{ii}^2}= 2\* R_{ii}(\bar{z})\* (R_N(\bar{z})\bar{v})_i\* (R_N(\bar{z})u)_i, \\
\label{109}
& \frac{ \partial (R_N(\bar{z})w)_i  }{\partial X_{ii}}=R_{ii}(\bar{z})\* (R_N(\bar{z})w)_i.
\end{align}
Using (\ref{102}) and (\ref{107}), one can write the last term in (\ref{dva2}) as 
\begin{align} 
\label{110}
& \frac{\sigma^2}{N}  \* \E \left(\sum_{i,j,k:i\not=k}  \bar{u_i}\*R_{kj}(z)\*v_j \* 
\frac{\partial \langle v,R_N(\bar{z}) u \rangle}{\partial X_{ik}}\right) =\\
\label{111}
& \frac{\sigma^2}{N}\* \E \left(\sum_{i,j,k:i\not=k} \bar{u_i}\*R_{kj}(z)\*v_j \* [(R_N(\bar{z})\bar{v})_i\* (R_N(\bar{z})u)_k + 
(R_N(\bar{z})\bar{v})_k\* (R_N(\bar{z})u)_i ]\right) =O\left(\frac{1}{|\Im z|^3\*N}\right).
\end{align}
The third cumulant terms in $r_N$ in (\ref{dva3}) can be written as
\begin{align}
\label{401}
& \frac{1}{2\* N^{3/2}} \* \sum_{i,j,k:i\not=k} \kappa_3(i,k)\*\bar{u_i}\* v_j\* \E \* 
\left(\frac{\partial^2(R_{kj}(z) \* \langle v,R_N(\bar{z}) u \rangle)}{\partial X_{ik}^2} \right) \\
\label{402}
& = \frac{1}{2\* N^{3/2}} \* \sum_{i,j,k:i\not=k} \kappa_3(i,k)\* \bar{u_i}\* v_j\* 
\E \* \left(\frac{\partial^2 R_{kj}(z)}{\partial X_{ik}^2} \* \langle v,R_N(\bar{z}) u \rangle \right) \\
\label{403}
& + \frac{1}{N^{3/2}} \*  \sum_{i,j,k:i\not=k} \kappa_3(i,k)\* \bar{u_i}\* v_j\*\E \* \left(\frac{\partial R_{kj}(z)}{\partial X_{ik}} \*
\frac{\partial  \langle v,R_N(\bar{z}) u \rangle}{\partial X_{ik}} \right) \\
\label{404}
& +  \frac{1}{2\* N^{3/2}} \* \sum_{i,j,k:i\not=k} \kappa_3(i,k)\* \bar{u_i}\* v_j\* 
\E \* \left(R_{kj}(z) \* \frac{\partial^2  \langle v,R_N(\bar{z}) u \rangle}{\partial X_{ik}^2} \right).
\end{align}

We are going to estimate the terms (\ref{402}-\ref{404}) separately.  We start with the last two.
We claim that both (\ref{403}) and (\ref{404}) are $O\left(\frac{1}{|\Im z|^4\*N}\right)$. 
Indeed, consider first (\ref{403}).  It follows from (\ref{vecher1}-\ref{vecher2}), (\ref{102}), and (\ref{107}),  that it is equal to
\begin{equation}
\label{405}
\frac{1}{2\* N^{3/2}} \* \sum_{i,j,k:i\not=k} \kappa_3(i,k)\*\bar{u_i}\* v_j\* \E \left( [R_{kk}(z)\*R_{ij}(z)+ R_{ik}(z)\*R_{kj}(z)]\* 
[(R_N(\bar{z})\bar{v})_i\* (R_N(\bar{z})u)_k + (R_N(\bar{z})\bar{v})_k\* (R_N(\bar{z})u)_i] \right).
\end{equation}
Let us estimate the term
\begin{equation}
\label{406}
\frac{1}{2\* N^{3/2}} \* \sum_{i,j,k:i\not=k} \kappa_3(i,k)\* \bar{u_i}\* 
v_j\* \E \left(R_{kk}(z)\*R_{ij}(z) \* (R_N(\bar{z})\bar{v})_i\* (R_N(\bar{z})u)_k \right)
\end{equation}
in (\ref{405}).  

We note that the Euclidean norm of the vector in $\C^N$ with the coordinates
$\kappa_3(i,k)\*u_i \* (R_N(\bar{z})\bar{v})_i, \ 1\leq i \leq N, i\not=k,$ and $0$ for $i=k$ is bounded from above by $ \frac{Const(m_5)}{|\Im z|}.$
Thus, it follows from (\ref{resbound}) and $\|v\|=1$ that
\begin{equation}
\label{408}
\big|\sum_{i,j:i\not=k} \kappa_3(i,k)\* \bar{u_i}\* (R_N(\bar{z})\bar{v})_i \* v_j\* R_{ij}(z) \big| \leq \frac{Const(m_5)}{|\Im z|^2}.
\end{equation}
In addition,
\begin{equation}
\label{407}
\sum_{k=1}^N |R_{kk}(z)|\* |(R_N(\bar{z})u)_k| \leq  \frac{1}{|\Im z|}\*\sum_{k=1}^N |(R_N(\bar{z})u)_k|=O\left(\frac{\sqrt{N}}{|\Im z|^2}\right).
\end{equation}

Combining (\ref{407}) and (\ref{408}), we estimate (\ref{406}) as $O\left(\frac{1}{|\Im z|^4\*N}\right).$ 
The other terms in (\ref{405}) can be estimated in a similar way, which implies that
(\ref{403}) is $O\left(\frac{1}{|\Im z|^4\*N}\right)$.  

Now, we turn our attention to (\ref{404}).
Using  (\ref{103}-\ref{105}) and (\ref{108}), one can rewrite (\ref{404}) as
\begin{align} 
\label{500}
& \frac{1}{N^{3/2}} \* \sum_{i,j,k:i\not=k} \kappa_3(i,k)\*\bar{u_i}\* v_j\* 
\E \* [R_{kj}(z) \* R_{ii}(\bar{z})\* (R_N(\bar{z})\bar{v})_k\* (R_N(\bar{z})u)_k ] \\
\label{501}
& + \frac{1}{N^{3/2}} \* \sum_{i,j,k:i\not=k} \kappa_3(i,k)\*\bar{u_i}\* v_j\* 
\E \* [R_{kj}(z) \*R_{ik}(\bar{z})\* (R_N(\bar{z})\bar{v})_i\* (R_N(\bar{z})u)_k ] \\
\label{502}
& + \frac{1}{N^{3/2}} \* \sum_{i,j,k:i\not=k} \kappa_3(i,k)\*\bar{u_i}\* v_j\* 
\E \* [R_{kj}(z) \* R_{ik} (\bar{z})\* (R_N(\bar{z})\bar{v})_k\* (R_N(\bar{z})u)_i ] \\
\label{503}
& + \frac{1}{N^{3/2}} \* \sum_{i,j,k:i\not=k} \kappa_3(i,k)\*\bar{u_i}\* v_j\* 
\E \* [R_{kj}(z) \*R_{kk} (\bar{z})\* (R_N(\bar{z})\bar{v})_i\* (R_N(\bar{z})u)_i].
\end{align}
We estimate (\ref{500}).  The subsums (\ref{501}-\ref{503}) can be estimated in a similar way.
The summation with respect to $j$ in (\ref{500}) gives
\begin{equation*}
\frac{1}{N^{3/2}} \* \sum_{i\not=k} \E \* [\kappa_3(i,k)\*\bar{u_i}\*(R_N(z) v)_k \*R_{ii}(\bar{z}) \* (R_N(\bar{z})\bar{v})_k\* (R_N(\bar{z})u)_k].
\end{equation*}
Now, we estimate
\begin{align}
& \sum_k |\kappa_3(i,k)\*(R_N(z) v)_k|\* |(R_N(\bar{z})\bar{v})_k| \* |(R_N(\bar{z})u)_k| \leq \frac{Const(m_5)}{|\Im z|^3}, \ \text{and} \nonumber \\
& \sum_i |u_i|\* |R_{ii}(\bar{z})| \leq \frac{\sqrt{N}}{|\Im z|}. \nonumber
\end{align}
Combining the last two bounds, we obtain that (\ref{500}) is $O\left(\frac{1}{|\Im z|^4\*N}\right).$ 

Finally, let us estimate (\ref{402}).  It can be written as 
\begin{align}
\label{Tret1}
& \frac{1}{2!\*N^{3/2}} \* 4\* \E \left(\sum_{i,j,k:i\not=k} 
\kappa_3(i,k)\*\bar{u_i}\*R_{ij}(z)\*R_{ik}(z)\*R_{kk}(z)\*v_j \*  \langle v,R_N(\bar{z}) u \rangle \right) \\
\label{Tret2}
& + \frac{1}{2!\*N^{3/2}} \* 2\*\E \left(\sum_{i,j,k:i\not=k} \kappa_3(i,k)\* 
\bar{u_i}\*R_{ii}(z)\*R_{kk}(z)\*R_{kj}(z)\*v_j \*\langle v,R_N(\bar{z}) u \rangle \right) \\
\label{Tret3}
& + \frac{1}{2!\*N^{3/2}} \* 2\*\E \left(\sum_{i,j,k:i\not=k} 
\kappa_3(i,k)\* \bar{u_i}\*R_{ki}(z)\*R_{ki}(z)\*R_{kj}(z)\*v_j \*\langle v,R_N(\bar{z}) u \rangle \right).
\end{align}
The subsums (\ref{Tret1}) and (\ref{Tret3}) are bounded from above by $O\left(\frac{1}{|\Im z|^4\*N}\right).$  
The calculations are very similar to the ones used above and are left to the reader.
The subsum (\ref{Tret2}) can be written as
\begin{equation*}
\E \left(\frac{1}{N^{3/2}}\* \sum_{i,j,k:i\not=k} \kappa_3(i,k)\* \bar{u_i}\*R_{ii}(z)\*R_{kk}(z)\*R_{kj}(z)\*v_j \* \overline{\eta_N}\right).
\end{equation*}
To estimate it, we write
\begin{align}
\label{200}
& \big|\E \big(\frac{1}{N^{3/2}}\* (\sum_{i,j,k:i\not=k} \kappa_3(i,k)\* \bar{u_i}\*R_{ii}(z)\*R_{kk}(z)\*R_{kj}(z)\*v_j) \* \overline{\eta_N})\big) - 
\frac{1}{N^{3/2}} \* \E (\sum_{i,j,k:i\not=k} \kappa_3(i,k)\* \bar{u_i}\*R_{ii}(z)\*R_{kk}(z)\*R_{kj}(z)\*v_j) \* \E \overline{\eta_N}\big| \\
\label{201}
& \leq \frac{1}{N^{3/2}} \* ( \V (\sum_{i,j,k:i\not=k} \kappa_3(i,k)\* \bar{u_i}\*R_{ii}(z)\*R_{kk}(z)\*R_{kj}(z)\*v_j))^{1/2} \* 
\left( \V ({\eta_N})\right)^{1/2}.
\end{align}
It follows from the estimates in (\ref{tretii2}) that one has a deterministic upper bound
\begin{equation*}
|\frac{1}{N^{3/2}} \sum_{i,j,k:i\not=k} \kappa_3(i,k)\* \bar{u_i}\*R_{ii}(z)\*R_{kk}(z)\*R_{kj}(z)\*v_j|\leq const \* \frac{1}{|\Im z|^3\*\sqrt{N}}.
\end{equation*}
Thus,
\begin{align}
\label{202}
& \E \left(\frac{1}{N^{3/2}}\* \sum_{i,j,k:i\not=k} \kappa_3(i,k)\* \bar{u_i}\*R_{ii}(z)\*R_{kk}(z)\*R_{kj}(z)\*v_j \* \overline{\eta_N})\right) \\
\label{203}
& = \frac{1}{N^{3/2}} \* \E (\sum_{i,j,k:i\not=k} \kappa_3(i,k)\* \bar{u_i}\*R_{ii}(z)\*R_{kk}(z)\*R_{kj}(z)\*v_j) \* \E \overline{\eta_N} 
+ O\left(\frac{1}{|\Im z|^3\*\sqrt{N}}\right)\*
\left( \V {\eta_N})\right)^{1/2}.
\end{align}
Combining the estimates (\ref{402}-\ref{203}), we obtain that
the third cumulant term (\ref{401}) contributing to $r_N$ in (\ref{dva1})  can be written as
\begin{align} 
\label{600}
& \frac{1}{N^{3/2}} \* \E (\sum_{i,j,k:i\not=k} \kappa_3(i,k)\*\bar{u_i}\*R_{ii}(z)\*R_{kk}(z)\*R_{kj}(z)\*v_j) \* \E \overline{\eta_N} \\
\label{601}
& + O\left(\frac{1}{|\Im z|^3\*\sqrt{N}}\right)\*
\left( \V {\eta_N})\right)^{1/2} + O\left(\frac{1}{|\Im z|^4\*N}\right).
\end{align}
Somewhat long but straightforward calculations using
(\ref{vecher1}-\ref{vecher2}) and (\ref{102}-\ref{111})
show that the fourth cumulant term in $r_N$ in (\ref{dva1}) can be estimated from above by
$O\left(\frac{1}{|\Im z|^5\*N}\right).\ $ Since the calculations are very similar to those in (\ref{chetvert}- \ref{chetvert2}), 
we leave the details to the reader.  In a similar fashion, the error terms in $r_N,$ due to the truncation of the decoupling formula
at $p=3$ for $i\not=k$ and at $p=1$ for $i=k$ are bounded from above 
by $O\left(\frac{1}{|\Im z|^6\*N}\right).\ $ The considerations are similar to those given in the analysis of
(\ref{trun}).  

Combining (\ref{amazon}), (\ref{110}-\ref{111}), (\ref{600}-\ref{601}), 
and the bounds on the fourth cumulant term and the error terms discussed in the above 
paragraph, one rewrites the Master equation (\ref{dva1}-\ref{dva2}) as
\begin{align}
\label{tri1000}
& z\* \E (|\eta_N|^2) =\langle u, v \rangle \* \overline{ \E \eta_N} +\sigma^2 \* \E \left( |\eta_N|^2 \* \tr_N R_N(z) \right) \\
\label{tri1}
& + \frac{\kappa_3}{N^{3/2}} \* \E (\sum_{ijk} \bar{u_i}\*R_{ii}(z)\*R_{kk}(z)\*R_{kj}(z)\*v_j) \* \E \overline{\eta_N} +
O\left(\frac{1}{|\Im z|^3\*\sqrt{N}}\right)\* \left( \V {\eta_N})\right)^{1/2}\\
\label{tri2}
& + O\left(\frac{P_6(|\Im z|^{-1})}{N}\right).
\end{align}
Using (\ref{shcherb}), we estimate
\begin{equation*}
|\E \left( |\eta_N|^2 \* \tr_N R_N(z) \right) - g_N(z) \*\E|\eta_N|^2| \leq \left(\V(|\eta_N|^2)\right)^{1/2}\* \left(\V \* \tr_N R_N(z)\right)^{1/2}=
O\left(\frac{1}{|\Im z|^4\*N}\right).
\end{equation*}
This allows us to write
\begin{align}
\label{trip1}
& z\* \E (|\eta_N|^2) = \langle u, v \rangle \* \overline{ \E \eta_N} +\sigma^2 \* g_N(z)\* \E |\eta_N|^2 
+ \frac{1}{N^{3/2}} \* \E (\sum_{i,j,k:i\not=k} \kappa_3(i,k)\*\bar{u_i}\*R_{ii}(z)\*R_{kk}(z)\*R_{kj}(z)\*v_j) \* \E \overline{\eta_N}\\
& + O\left(\frac{1}{|\Im z|^3\*\sqrt{N}}\right)\* \left( \V {\eta_N})\right)^{1/2} + O\left(\frac{P_6(|\Im z|^{-1})}{N}\right).\nonumber
\end{align}
Subtracting  the r.h.s. in (\ref{twotwo}) from the r.h.s. in (\ref{trip1}), we obtain (\ref{101}).
Lemma \ref{LemmaVar} is proven.
\end{proof}
Now, we are ready to finish the proof of Proposition \ref{proposition:prop1}.
To obtain the estimate (\ref{peredachavar}) from (\ref{101}), we use the same arguments as in Section 3 of \cite{ORS} and Section 2 of \cite{PRS}.  
We note (see e.g. (3.9) in \cite{ORS}) that
\begin{equation}
\label{vasil}
g_N(z)\*(z-\sigma^2\*g_N(z))=1+ O\left( \frac{P_4(|\Im z|^{-1})}{N}\right).
\end{equation} 
We define 
\begin{equation*}
\mathcal{O}_N:= \{z: |\Im z|> L\*N^{-1/4} \},
\end{equation*}
where the constant $L$ is chosen sufficiently large so that the 
$O\left( \frac{P_4(|\Im z|^{-1})}{N}\right)$ term on the r.h.s. of (\ref{vasil}) is at most $1/2$ in absolute value.
Multiplying both sides of (\ref{101}) by $g_N(z),$  and using (\ref{STbound}), we obtain that
\begin{align}
\label{1011}
& \V( \langle u^{(N)},R_N(z) v^{(N)} \rangle ) = 
\sqrt{ \V( \langle u^{(N)},R_N(z)\* v^{(N)} \rangle )} \* O\left( \frac{P_4(|\Im z|^{-1})}{\sqrt{N}} \right) \\
\label{1012}
& + O\left( \frac{P_7(|\Im z|^{-1})}{N} \right),  \nonumber 
\end{align}
for $z \in \mathcal{O}_N. $
It follows from (\ref{1011}) that
\begin{equation}
\label{610}
\V( \langle u^{(N)},R_N(z) v^{(N)} \rangle )= O\left( \frac{P_8(|\Im z|^{-1})}{N} \right) \ \text{for} \ z \in \mathcal{O}_N.
\end{equation}
On the other hand, if $ |\Im z| \leq L\*N^{-1/4},$ then $\frac{L^4}{N\*|\Im z|^4} \geq 1.$ 
Since $|\langle u^{(N)},R_N(z) v^{(N)} \rangle| \leq  \frac{1}{|\Im z|}, $ we have
\begin{equation}
\label{611}
\V( \langle u^{(N)},R_N(z) v^{(N)} \rangle ) \leq \frac{1}{|\Im z|^2} \leq \frac{L^4}{N \*|\Im z|^6}, 
\end{equation}
for $z$ such that $|\Im z| \leq L\*N^{-1/4}.$
Combining (\ref{610}) and (\ref{611}), we obtain (\ref{peredachavar}).
This finishes the proof of Proposition \ref{proposition:prop1}.
\end{proof}
\section{ \bf{Proof of Theorem 1.6}}
\label{sec:proofquadratic}
\begin{proof}
Our exposition follows closely the ones in Section 3 of \cite{PRS} and Section 4 of \cite{ORS}.
In order to extend the estimates of Proposition \ref{proposition:prop1}
to a more general class of test functions, we use
the Helffer Sj\"{o}strand functional calculus (see \cite{HS}, \cite{D}).

Let $l$ be some non-negative integer, and
$f \in C^{l+1}(\mathbb{R})$ decay at infinity sufficiently fast.
For any self-adjoint operator $X$ we can write
\begin{equation}
 f(X)=-\frac{1}{\pi}\,\int_{\mathbb{C}}\frac{\partial \tilde{f}}{\partial \bar{z}}\, \frac{1}{z-X}\,dxdy 
\quad,\quad\frac{\partial \tilde{f}}{\partial \bar{z}} := \frac{1}{2}\Big(\frac{\partial \tilde{f}}
{\partial x}+i\frac{\partial \tilde{f}}{\partial y}\Big),
\label{formula-H/S}
 \end{equation}
 where:
 \begin{itemize}
\item[i)]
 $z=x+iy$ with $x,y \in \mathbb{R}$;
 \item[ii)] $\tilde{f}(z)$ is the extension of the function $f$ defined as follows
  \begin{equation}\label{a.a. -extension}
  \tilde{f}(z):=\Big(\,\sum_{n=0}^{l}\frac{f^{(n)}(x)(iy)^n}{n!}\,\Big)\sigma(y);
\end{equation}
 here $\sigma \in C^{\infty}(\mathbb{R})$ is a nonnegative function equal to $1$ for $|y|\leq 1/2$ and equal to zero for $|y|\geq 1$.
 \end{itemize}
The integral in (\ref{formula-H/S}) does not depend on the choice of $l$ and the cut-off function (see e.g. \cite{D}).
Using the definition of $\tilde{f}$ in (\ref{a.a. -extension}) one can easily calculate 
\begin{eqnarray}
\frac{\partial \tilde{f}}{\partial \bar{z}}&=&\frac{1}{2}\Big(\frac{\partial \tilde{f}}{\partial x}+i\frac{\partial \tilde{f}}{\partial y}
\Big)\\
& =&\frac{1}{2} \Big(\,\sum_{n=0}^{l}\frac{f^{(n)}(x)(iy)^n}{n!}\,\Big)
i\frac{d\sigma}{dy}+\frac{1}{2}f^{(l+1)}(x)(iy)^l\frac{\sigma(y)}{l!}
\end{eqnarray}
and derive the crucial bound
\begin{equation}\label{estimate-derivative}
\Big|\frac{\partial \tilde{f}}{\partial \bar{z}} (x+iy)\Big|\leq  C_1\* \max\left(|\frac{d^jf}{dx^j}(x)|, \ 1\leq j \leq l+1\right) \*
|y|^l\quad.
\end{equation}
 
For $X=X_N$,  (\ref{formula-H/S}) implies
 \begin{equation}
\label{integral1}
\langle u, f(X_N) v \rangle=-\frac{1}{\pi}\,\int_{\mathbb{C}}\frac{\partial \tilde{f}}{\partial \bar{z}}  \langle u,
R_N(z) v \rangle dxdy. \end{equation}
To prove (\ref{ska2}), we let $l=7$ in (\ref{a.a. -extension}) and assume that $f$ has compact support. It follows from (\ref{peredacha}) that 
\begin{eqnarray}\label{zato}
& &\mathbb{E}\langle u, f(X_N) v \rangle=-\mathbb{E} \frac{1}{\*\pi\* } \* \int_{\mathbb{C}} \frac{\partial \tilde{f}}{\partial \bar{z}}  \* 
\langle u, R_N(z) v \rangle \*
dxdy\\
&= &
-\frac{1}{\pi}\, \langle u, v \rangle \*\,\int_{\mathbb{C}}\frac{\partial \tilde{f}}{\partial \bar{z}}\,   \* g_{\sigma}(z) \* dxdy
-\frac{1}{\pi}\, \,\int_{\mathbb{C}}\frac{\partial \tilde{f}}{\partial \bar{z}}\,  \* \epsilon_{u,v}(z) \* dxdy \label{main equation}\\
&= &
\langle u, v \rangle \*\,\int f(x) d\mu_{sc}(x)
-\frac{1}{\pi}\, \,\int_{\mathbb{C}}\frac{\partial \tilde{f}}{\partial \bar{z}}\,  \* \epsilon_{u,v}(z) \* dxdy\end{eqnarray}
where 
\begin{equation}
\label{spartalporto}
| \epsilon_{u,v}(z)|\leq C_2 \* \frac{1}{\sqrt{N}}\,\frac{1}{|Imz|^7},
\end{equation}
uniformly on $\{z: \Re z \in supp(f), \ |\Im z|\leq 1\},$ and $C_2$ is a constant depending on $supp(f).$
\noindent
We conclude that the second term on the r.h.s. of (\ref{main equation}) can be estimated as follows
\begin{align}
& \big|\frac{1}{\pi}\, \,\int_{\mathbb{C}}\frac{\partial \tilde{f}}{\partial \bar{z}}\,  \* \epsilon_{u,v}(z) \* dxdy \big| 
\leq  \frac{1}{\pi}\, \,\int_{\mathbb{C}}|\frac{\partial \tilde{f}}{\partial \bar{z}}\,  \* \epsilon_{u,v}(z) \*| dxdy\\
& \leq C_1 \*C_2  \* \|f\|_{C^8([-L,L])} \*\frac{1}{\sqrt{N}} 
\int dx \chi_f(x) \int dy \chi_\sigma(y)
\end{align}
where $\chi_f$  and $\chi_{\sigma}$ are the characteristic functions of the support of $f$ and of $\sigma$ respectively, and $L$ is such that
$supp(f)\subset [-L, L].$
This proves (\ref{ska2}). 

To prove (\ref{ska2p}), one considers $f \in C^9(\R)$ (so $l=8$) such that
$\|f\|_{9,1,+}$ is finite.  Using (\ref{peredachauh}), one replaces the estimate (\ref{spartalporto}) with
\begin{equation}
\label{spartalporto2}
| \epsilon_{u,v}(z)|\leq C_3 \* \frac{|z|+M}{\sqrt{N}}\* P_8(|Imz|^{-1}),
\end{equation}
valid on $\C \setminus \R,$ which leads to
\begin{align}
& |\frac{1}{\pi}\, \,\int_{\mathbb{C}}\frac{\partial \tilde{f}}{\partial \bar{z}}\,  \* \epsilon_{u,v}(z) \* dxdy| 
\leq  \frac{1}{\pi}\, \,\int_{\mathbb{C}}|\frac{\partial \tilde{f}}{\partial \bar{z}}\,  \* \epsilon_{u,v}(z) \*| dxdy
\leq C_1\*C_3 \*\|f\|_{9,1,+}\*\frac{1}{\sqrt{N}}.
\end{align}

To prove (\ref{ska1}), we consider $f\in C^5(\R)$ such that  $\|f\|_{5,1}<\infty,$ and let $l=4$ in (\ref{a.a. -extension}).  Then
\begin{align}
& \V (\langle u, f(X_N)\* v \rangle )= \V \left( -\frac{1}{\pi}\,\int_{\mathbb{C}}
\frac{\partial \tilde{f}}{\partial \bar{z}} \langle u, R_N(z)\* v \rangle \*  dxdy \right) \nonumber \\
& =\frac{1}{\pi^2} \,\int_{\mathbb{C}}\,\int_{\mathbb{C}} \frac{\partial \tilde{f}}{\partial \bar{z}} \*
\frac{\partial \tilde{f}}{\partial \bar{w}} \Cov \left(\langle u, R_N(z)\*v \rangle, \langle u, R_N(w)\*v \rangle\right) \*dxdydsdt, \nonumber
\end{align}
where $z=x+iy, \ w=s+it.$
Taking into account (\ref{peredachavar}), we get
\begin{align}
& \V (\langle u, f(X_N)\*v \rangle) \leq \frac{1}{\pi^2}\,\int_{\mathbb{C}}\,\int_{\mathbb{C}}|\frac{\partial \tilde{f}}{\partial \bar{z}}|\* 
|\frac{\partial \tilde{f}}{\partial \bar{w}}| \* \sqrt{\V(\langle u, R_N(z)\*v \rangle)}\* \sqrt{\V(\langle u, R_N(w)\*v \rangle)}  dxdydsdt  \nonumber\\
\label{234}
& \leq \frac{Const}{N} \* \left(\int_{\mathbb{C}}\, \big|\frac{\partial \tilde{f}}{\partial \bar{z}}\big|\*  P_4(|\Im z|^{-1}) dxdy \right)^2. 
\end{align}
Plugging (\ref{estimate-derivative}) with $l=4$ in (\ref{234}), we prove (\ref{ska1}).  
Thus, we have proved the parts (i) and (ii) of Theorem \ref{thm:quadratic}.

Now, let us assume that the marginal distributions of the entries of $W_N$ satisfy the Poincar\'e inequality (\ref{poin}) 
with a uniform constant $\upsilon$ and prove the parts (iii)-(v), i.e. the estimates
(\ref{ska3}), (\ref{ska4}), and (\ref{ska5}).  Since the proof of 
(\ref{ska3}-\ref{ska5}) is very similar to the proof of Proposition 3.3 in \cite{PRS}, we discuss here only the main ingredients.

The first important observation is that if $f(x)$ is a Lipschitz continuous function on $\R$ with the Lipschitz constant $|f|_{\mathcal{L}, \R}$
then on the space of the $N\times N$ real symmetric (Hermitian) matrices, the matrix-valued function $f(X)$ is also 
Lipschitz continuous with respect to the Hilbert-Schmidt norm (\cite{CQT}, Proposition 4.6, c)).  Namely,
\begin{equation}
\label{villa}
\|f(X)-f(Y)\|_{HS}\leq |f|_{\mathcal{L},\R} \* \|X-Y\|_{HS},
\end{equation}
where the Hilbert-Schmidt norm is defined as
\begin{equation}
\|X\|_{HS}=\left(\Tr (|X|^2)\right)^{1/2}.
\end{equation}
In particular, if $u$ and $v$ are unit vectors, then 
\begin{equation}
\label{bragaa}
G(X_N)=\langle u, f(X_N)\*v \rangle)
\end{equation}
is a complex-valued Lipschitz continuous function
on the space of $N\times N$ real symmetric (Hermitian) matrices with the Lipschitz constant 
\begin{equation*}
|G|_{\mathcal{L}}:=sup_{X\not=Y} \frac{|G(X)-G(Y)|}{\|X-Y\|_{HS}}= |f|_{\mathcal{L}, \R}.
\end{equation*}
The second observation is that joint distribution of the matrix entries 
$$ \{ X_{ii}, \ 1\leq i \leq N, \ \sqrt{2}\*X_{jk}, \ 1\leq j<k\leq N \} $$ of $X_N$ satisfies the Poincar\'e inequality
with the constant $\frac{1}{2}\*N\*\upsilon$ since the Poincar\'e inequality tensorizes (\cite{GZ}, \cite{AGZ}).
Therefore, for any complex-valued Lipschitz continuous function of the matrix entries with the Lipschitz constant
$|G|_{\mathcal{L}},$
the distribution of $G(X_N)$ has exponential tails (see e.g. Lemma 4.4.3 and Exercise 4.4.5 in \cite{AGZ}), i.e.
\begin{equation}
\label{BOLUKLON}
\P\left(|G(X_N)-\E G(X_N)|\geq t \right) \leq 2\*K\* \exp\left(-\frac{\sqrt{\upsilon\*N}}{2\*|G|_{\mathcal{L}}}\*t\right),
\end{equation}
where $K$ is a universal constant,
\begin{equation*}
K=-\sum_{i\geq 0} 2^i\*\log(1-2^{-1}\*4^{-i}).
\end{equation*}
This proves (\ref{ska4}).

Applying (\ref{BOLUKLON}) to the spectral norm $\|X\|$ of the matrix $X_N$ and using the universality results for the largest 
eigenvalues (see \cite{J} and references therein),
we obtain
\begin{equation}
\label{bol}
\P\left(| \|X_N\|- 2\*\sigma|\geq t \right) \leq (2\*K +o(1))\* \exp\left(-\frac{\sqrt{\upsilon\*N}}{2}\*t\right).
\end{equation}
and, in particular,
\begin{equation}
\label{bolukl}
\P (\|X_N\|> 2\*\sigma +\delta) \leq (2\*K+o(1))\* \exp\left(-\frac{\sqrt{\upsilon\*N}}{2}\*\delta\right).
\end{equation}

Let $f(x)$ be a real-valued  Lipschitz continuous function on $[-2\sigma-\delta, 2\sigma+\delta].$  Then, we can find a  function
$f_1(x)$ that is Lipschitz continuous on $\R,$ coincides with $f$ on $[-2\sigma-\delta, 2\sigma+\delta],$ and satisfies
$|f_1|_{\mathcal{L},\R}=|f|_{\mathcal{L},\delta}.$ It follows from (\ref{bolukl}) that $\langle u, f(X_N)\*v \rangle)$ does not coincide with
$\langle u, f_1(X_N)\*v \rangle)$ on a set of probability at most $(2\*K +o(1))\* \exp\left(-\frac{\sqrt{\upsilon\*N}}{2}\*t\right),$ which implies 
(\ref{ska3}).  The details are left to the reader.
\end{proof}
\section{\bf{Outliers in the Spectrum of Finite Rank Perturbations of Wigner Matrices}}
\label{sec:outliers}
This section is devoted to the proof of Theorem \ref{thm:fluctoutliers}
\begin{proof}
For $x \in (2\*\sigma, +\infty),$
\begin{equation}
\label{realst}
g_{\sigma}(x)=\frac{x}{2\*\sigma^2}\left(1-\sqrt{1-\frac{4\*\sigma^2}{x^2}}\right)
\end{equation}
is decreasing and $ \ g_{\sigma}(2\*\sigma+0)=1/\sigma.$
Let us choose $\delta>0 $ in such a way that 
\begin{equation}
\label{juar}
\theta_j > \frac{1}{g_{\sigma}(2\*\sigma +2\*\delta)} \  \text{for all} \ 1\leq j \leq J_{\sigma^+},
\end{equation}
i.e. for all $\theta_j $ that correspond to the outliers (so $\theta_j >\sigma$).
Let 
\begin{equation*}
L:= \max(\theta_j, \ 1\leq j \leq J_{\sigma^+}) + 2\*\sigma +2\*\delta.
\end{equation*}
It follows from (\ref{offdiagreal1}-\ref{diagreal2}) (see e.g. \cite{BYin}, \cite{B}, and the proof of Proposition 2.1 in \cite{ORS})
that there exists a random real symmetric Wigner matrix $\tilde{W_N}$ that satisfies (\ref{offdiagreal1}-\ref{diagreal2}),
\begin{equation*}
\P(W_N=\tilde{W_N})\to 1 \ \text{as} \ N\to \infty,
\end{equation*}
and
\begin{equation}
\label{normag}
\|\tilde{W_N}/\sqrt{N}\|\to 2\*\sigma \ \text {a.s.}
\end{equation}
Without loss of generality, we can assume that $\tilde{W_N}=W_N,$ so
\begin{equation}
\label{norma}
\|W_N/\sqrt{N}\|\to 2\*\sigma \ \text {a.s.}
\end{equation}
It follows from the definition of $M_N$ and (\ref{norma}) that,
with probability $1$, the deformed random matrix 
$M_N=\frac{1}{N}\*W_N +A_N $ has no eigenvalues bigger than $L\ $ for all but finitely many $N.$
Let $u=u^{(N)}, $ and $v=v^{(N)}$ be nonrandom unit vectors in $\C^N.$
Define
\begin{align}
\label{xiN}
& \xi_N(x):=\sqrt{N}\* \left(\langle u, R_N(x) \* v \rangle - g_{\sigma}(x) \* \langle u, v \rangle \right), 
\ x\in [2\*\sigma +2\*\delta, \infty), \\
\label{zetaN}
& \zeta_N(x):= \frac{d \*\xi_N(x)}{dx} = -\sqrt{N}\* \left(\langle u, R^2_N(x) \* v \rangle +g'_{\sigma}(x) \* \langle u, v \rangle \right)   , 
\ x\in [2\*\sigma +2\*\delta, \infty), \\
\label{xiNt}
& \tilde{\xi}_N(x):=\sqrt{N}\* \left(\langle u, (h(X_N)\*R_N(x)) v \rangle - g_{\sigma}(x) \* \langle u, v \rangle \right), 
\ x\in [2\*\sigma +2\*\delta, \infty), \\
\label{zetaNt}
& \tilde{\zeta}_N(x):= 
\frac{d \*\tilde{\xi}_N(x)}{dx} = -\sqrt{N}\* \left(\langle u, (h(X_N)\*R^2_N(x)) v \rangle +g'_{\sigma}(x) \* \langle u, v \rangle \right)   , 
\ x\in [2\*\sigma +2\*\delta, \infty),
\end{align}
where $h\in C^{\infty}_c(\R)$ such that
\begin{align}
\label{h1}
& h(x)\equiv 1 \ \text{for} \ x\in [-2\*\sigma -\delta/2, 2\*\sigma +\delta/2], \\
\label{h2}
& h(x) \equiv 0 \ \text{for} \ x\not\in [-2\*\sigma -\delta, 2\*\sigma +\delta].
\end{align}
We claim the following lemma.
\begin{lemma}
\label{Lemma11}
\begin{equation}
\label{cvet}
\P \left(\max \left(  | \zeta_N(x) |, \ x \in [2\*\sigma+2\*\delta, L] \right) \leq \log (N) \*N^{1/6} \right) \to 1,
\end{equation}
where $\zeta_N(x)$ is defined in (\ref{zetaN}).
\end{lemma}
\begin{proof}

It follows from (\ref{norma}) that $\xi_N(x)=\tilde{\xi}_N(x) $ and $\zeta_N(x)=\tilde{\zeta}_N(x)$ for all $x \in [2\sigma+2\delta, L]$ 
and all but finitely many $N$ almost surely.  Thus, it is enough to prove the result
of the lemma for $\tilde{\zeta}_N(x).$
Consider an equidistributed finite sequence 
$$x_0=2\*\sigma+2\*\delta < x_1 <x_2 <\ldots  < x_{l(N)}, $$ 
where 
$x_{i+1}-x_i =N^{-1/3}, \ 0\leq i \leq l(N)-1, \ $ and $ x_{l(N)-1}\leq L < x_l(N). $  Clearly, the number of elements in the 
sequence is $O(N^{1/3}). $
We have
\begin{align}
\label{noch1}
& \P \left \{ \max\left( |\tilde{\zeta}_N(x_i)|, \ 0\leq i \leq l(N)-1 \right) > \frac{1}{2}\*\log (N) \*N^{1/6} \right\} \\
\label{noch2}
& \leq \sum_{i=0}^{l(N)} \* \P \left \{|\tilde{\zeta}_N(x_i)| > \frac{1}{2}\*\log (N) \* N^{1/6} \right\}  \\
\label{noch3}
&  \leq \frac{1}{(\log N)^2\* N^{1/3}} \*
\sum_{i=0}^{l(N)}   \* \left( \V (\tilde{\zeta}_N(x_i)) + \left(\E(\tilde{\zeta}_N(x_i))\right)^2 \right).
\end{align}
It follows from Theorem \ref{thm:quadratic} that 
\begin{align}
\label{davis1}
& \V (\tilde{\zeta}_N(x_i)) = O(1), \\
\label{davis2}
& \E(\tilde{\zeta}_N(x_i)) =  O(1),
\end{align}
uniformly in $0\leq i\leq l(N)$ and $N\geq 1.\ $ Indeed,
\begin{equation}
\label{mnogo10}
\tilde{\zeta}_N(x)= 
\sqrt{N}\* \left(\langle u, f^{(x)}(X_N) v \rangle - \langle u, v \rangle \* \int_{-2\sigma}^{2\sigma} f^{(x)}(t) \*d\mu_{sc}(t)  \right),
\end{equation}
where
\begin{equation}
\label{mnogo11}
f^{(x)}(t)= -h(t)\* \frac{1}{(x-t)^2}
\end{equation}
is a $C^{\infty}_c(\R)$ function such that 
$\|f^{(x)}\|_{5,1} $ and 
$\|f^{(x)}\|_{C^8_c(\R)}$ are uniformly bounded in 
$x\in [2\*\sigma +2\*\delta, \infty).\ $ Thus, (\ref{davis1}-\ref{davis2}) follow from (\ref{ska1}-\ref{ska2}).
The bounds (\ref{noch1}-\ref{noch3}) and (\ref{davis1}-\ref{davis2}) then imply
\begin{equation}
\label{noch4}
\P \left \{ \max\left( |\tilde{\zeta}_N(x_i)|, \ 0\leq i \leq l(N)-1 \right) > \frac{1}{2}\*\log (N) \*N^{1/6} \right\} \leq \frac{const}{(\log N)^2}.
\end{equation}
Taking into account that $|\frac{d\*\tilde{\zeta}_N(x_i)}{dx}| \leq const \* \sqrt{N} \* \|u\|\*\|v\|, \ $ we arrive at (\ref{cvet}).
Lemma \ref{Lemma11} is proven.
\end{proof}
Now, we are ready to start the proof of Theorem \ref{thm:fluctoutliers}.
Let us denote by $u^{(1)}, \ldots, u^{(r)}, $ the orthonormal eigenvectors of $A_N$ corresponding to the non-zero eigenvalues.
We recall that we used the notation 
$\theta_1 > \ldots > \theta_{j_0}=0 > \ldots  > \theta_J$  for the (fixed) eigenvalues 
of $A_N,$ and denoted the (fixed) multiplicity of $\theta_j$ by $k_j$.  The zero eigenvalue $\theta_{j_0}=0 $ has multiplicity $N-r.$ Clearly,
$ \sum_{j\neq j_0} k_j =r.$  Let us denote by $\Theta$ the $r\times r$ diagonal matrix built from the non-zero eigenvalues of $A_N,$
\begin{equation}
\label{thetat}
\Theta:= diag(\theta_1, \ldots, \theta_1, \ldots,  \theta_{j_0-1}, \ldots, \theta_{j_0-1}, \theta_{j_0+1}, \ldots, \theta_{j_0+1}, \ldots, 
\theta_J, \ldots, \theta_J).
\end{equation}
Let us also denote by $U_N $ the $N \times r$ matrix whose columns are given by the orthonormal eigenvectors $u^{(1)}, \ldots, u^{(r)} $ of $A_N.$
Clearly,
\begin{equation}
A_N=U_N \* \Theta \* U_N^*.
\end{equation}
For any $x\in [2\*\sigma+2\delta, L], $ we define the $r\times r$ matrix $\Xi_N(x)$ as follows. Let
\begin{equation}
\label{yalta}
(\Xi_N(x))_{ij}= \xi^{ij}_N(x):=\sqrt{N}\* \left(\langle u^{(i)}, R_N(x) u^{(j)}\rangle - g_{\sigma}(x)\*\delta_{ij} \right), \ 1\leq i,j \leq r.
\end{equation}     

The first step in the proof of Theorem \ref{thm:fluctoutliers} is the following lemma from \cite{BGM}.
\begin{lemma}
\label{Lemma12}
Suppose that $x$ is not an eigenvalue of $X_N.$  Then $x$ is an eigenvalue of $X_N + A_N$ with multiplicity
$n\geq 1$ if and only if $g_{\sigma}(x)$ is an eigenvalue of the $r\times r$ matrix
\begin{equation}
\label{ZN}
Z_N(x):=\Theta^{-1} -\frac{1}{\sqrt{N}} \* \Xi_N(x),
\end{equation}
with the same multiplicity.
\end{lemma}
For the convenience of the reader, we sketch the proof of Lemma \ref{Lemma12} below.
\begin{proof}
Let $ x \not\in Sp(X_N). $ Therefore $R_N(x)= (x\*I_N -X_N)^{-1} $ is well defined, and
\begin{equation}
\det(X_N +A_N - x\* I_N)=\det \left( (X_N-x \* I_N)\* (Id_N  -R_N(x) \*A_N) \right).
\end{equation}
We obtain that for $ x \not\in Sp(X_N)$ that $ \ x \in Sp(X_N +A_N) $ if and only if 
\begin{equation}
\det \left(I_N - R_N(x) \*A_N \right)= 
\det\left(I_N - R_N(x) \*U_N \* \Theta \* U_N^* \right)= \det \left( I_r - \Theta \* U_N^* \* R_N(x)\*U_N \right) =0,
\end{equation}
where one uses the identity $ \det (I -B\*C)= \det(I -C\*B).\ $  Rewriting
$$I_r - \Theta\* U_N^* \* R_N(x)\*U_N= -\Theta \* \left(U_N^* \* R_N(x)\*U_N - \Theta^{-1} \right), $$ 
one finishes the proof of Lemma \ref{Lemma12}.
\end{proof}

Proposition \ref{Prop1} plays an important role in the proof of Theorem \ref{thm:fluctoutliers}.  Before we prove
Proposition \ref{Prop1}, we need to introduce some notations and prove Lemma \ref{Lemma14}.

Consider a family of $r \times r$ matrices $Z_N(x)$ defined in (\ref{ZN}) for $x \in [2\*\sigma +2\delta, L]. $ 
Fix an eigenvalue $\theta_j $ of $A_N$ such that $\theta_j> \sigma$ and use the notation $v^{(1)}, \ldots, v^{(k_j)} $ for the eigenvectors of 
$A_N$ that correspond to the 
eigenvalue $\theta_j. \ $  
Without loss of generality we can assume that $j=1.$  We do it just to simplify notations.  The case $1 <j\leq J_{\sigma^+} $ is identical.
We recall that $ \Xi^{(j)}_N $ is defined in (\ref{thetamatrica}) as the 
$k_j\times k_j $ submatrix of $\Xi_N(\rho_j) $ restricted to the rows and columns corresponding to $v^{(i)}, \ 1\leq i \leq k_j.$ 
The central role in the proof of Proposition \ref{Prop1} is played by the following lemma.
\begin{lemma}
\label{Lemma14}
Let $Z_N(x), \ x \in [2\*\sigma +2\delta, L], $ be as in (\ref{ZN}), with $\Xi_N(x)$ defined in (\ref{yalta}), and $\Theta$ defined in (\ref{thetat}).
Let 
\begin{equation}
\label{hart}
z_1(x)\leq z_2(x)\leq \ldots \leq z_r(x)
\end{equation}
be the ordered eigenvalues of $Z_N(x).$  
Then, for sufficiently large constant $C>0, $
\begin{equation}
\label{maroc}
\P \left(|z_i(x)-z_i(y)| \leq C \* \frac{\log N}{N^{1/3}}\*|x-y|, \ \forall x,y \in [2\*\sigma +2\delta, L], 
\ i=1, \ldots, r \right) \to 1,
\end{equation}
as $N \to \infty, $ and
\begin{equation}
\label{tunis}
z_i(\rho_1)= \frac{1}{\theta_1} + O(\frac{1}{\sqrt{N}}), \ 1\leq i \leq k_1,
\end{equation}
in probability, i.e. $\sqrt{N}\*(z_i(\rho_1)-\frac{1}{\theta_1}) $ is bounded in probability, $ \ 1\leq i \leq k_1.$
\end{lemma}
Below, we prove Lemma \ref{Lemma14}.
\begin{proof}
We claim that (\ref{maroc}) follows from  Lemma \ref{Lemma11}.  Indeed, (\ref{cvet}) and (\ref{zetaN}) imply that
\begin{equation}
\label{jj}
\P \left(\|\Xi_N(x) - \Xi_N(y)\|  \leq Const \* \log(N)\*N^{1/6}\*|x-y|, \ \forall x,y \in [2\*\sigma +2\delta, L] \right) \to 1,
\end{equation}
as $N \to \infty.$  Since $ |z_i(x)-z_i(y)| \leq \|Z_N(x)-Z_N(y)\|=\frac{1}{\sqrt{N}} \* \|\Xi_N(y) - \Xi_N(y)\|, \ 1\leq i \leq r,  $ we conclude that 
(\ref{jj}) implies (\ref{maroc}).

To prove (\ref{tunis}), we use the fact that 
\begin{equation}
\label{andor}
\|Z_N(x)- \Theta^{-1} \|= \| \frac{1}{\sqrt{N}} \* \Xi_N(x)\| =  O(\frac{1}{\sqrt{N}}), 
\end{equation}
in probability.  
Indeed, the entries of the $r\times r$ matrix $\Xi_N(x)$ are bounded in probability 
since the expectation and variance of 
\begin{equation*}
\tilde{\xi}^{ij}_N(x):=\sqrt{N}\* \left(\langle u^{(i)}, h(X_N)\*R_N(x) u^{(j)}\rangle - g_{\sigma}(x)\*\delta_{ij} \right), \ 1\leq i,j \leq r,
\end{equation*}
is uniformly bounded by Theorem \ref{thm:quadratic},
and
\begin{equation*}
\tilde{\xi}^{ij}_N(x)= \xi^{ij}_N(x), \ 1\leq i,j \leq r,
\end{equation*}
almost surely. Thus, $\|\Xi_N(x)\|$ is also bounded in 
probability.  Since the first $k_1$ eigenvalues of $\Theta^{-1}$ are equal to $\frac{1}{\theta_1}, $ we obtain (\ref{tunis}).
Lemma \ref{Lemma14} is proven.
\end{proof}
Now, we are ready to prove Proposition \ref{Prop1}.
\begin{proof}
By Lemma \ref{Lemma12}, the outliers of $X_N+A_N$ are given by those values of \\$x \in [2\*\sigma+\delta, M]$
such that 
\begin{equation}
\label{algeria}
g_{\sigma}(x)= z_i(x) \ \text{ for some}  \ 1\leq i \leq r.
\end{equation}
We recall that $g_{\sigma}(x)$ is a monotonically decreasing function on $[2\*\sigma+\delta, M]$ and 
\begin{equation}
\label{libya}
g_{\sigma}'(x) \leq const(\sigma, \delta, M)<0, \ x \in [2\*\sigma+\delta, M].
\end{equation}
We note that 
\begin{equation}
\label{tictac}
g_{\sigma}(\rho_1)=\frac{1}{\theta_1}.
\end{equation}
Since for  $1\leq i \leq k_1, $ (\ref{tunis}) gives us that $z_i(\rho_1)- g_{\sigma}(\rho_1) = O(\frac{1}{\sqrt{N}})$ in probability,
it follows from (\ref{algeria}) and (\ref{maroc}) that with probability going to $1,$
there exist 
$M > x_1 \geq x_2 \geq \ldots \geq x_{k_1}> 2\*\sigma +\delta$ such that
$g_{\sigma}(x_i)=z_i(x_i), \ 1\leq i \leq k_1,\ $  
and 
\begin{equation}
\label{madagascar}
\sqrt{N} \*|x_i-\rho_1| =O(1), \ 1\leq i \leq k_1,
\end{equation}
in probability.
Applying (\ref{maroc}) one more time, we get that
\begin{equation}
g_{\sigma}(x_i)=z_i(\rho_1) + O\left(\frac{\log (N)\*N^{1/6}}{N}\right)   , \ 1\leq i \leq k_1,
\end{equation}
in probability.  By a standard perturbation theory argument (see e.g. section XII.1 in \cite{RS4}), one proves that the first 
$k_1$ smallest eigenvalues of the matrix $Z_N(\rho_1)$ differ from the 
(increasingly ordered) eigenvalues of the $k_1\times k_1$ matrix $\frac{1}{\theta_1}\*Id -\frac{1}{\sqrt{N}} \* \Xi^{(m)}_N$ by at most 
$O\left(\frac{1}{N}\right),\ $ in probability, where the matrix $\Xi^{(m)}_N$ has been defined in (\ref{thetamatrica}). 
To see this, we use the following standard lemma from the perturbation theory

\begin{lemma}
\label{Lemma177}
Let $B$ be an $n\times n$ real symmetric (Hermitian) matrix that can be written in the block form as $B=(B_{ij})_{i,j=1,2},$ where 
$B_{ij}$ is an $n_i\times n_j$ matrix. Suppose that all eigenvalues of $B_{11}$ are smaller than all eigenvalues of $B_{22}$ and the gap between the 
spectra  of $B_{11}$ and $B_{22}$ is at least $Const>0.$   In addition, suppose that the operator norm of the offdiagonal block $B_{12}$ 
is bounded from above by $\epsilon,$ so that $\|B_{12}\|=\|B_{21}\|\leq \epsilon.$  

Then there exists $const(Const, n)$ such that the first $n_1$ smallest eigenvalues of $B$ differ from the (increasingly ordered) 
eigenvalues of $B_{11}$ by at most $ const\* \epsilon^2.$
\end{lemma}
\begin{proof}
We sketch the main idea of the proof for the convenience of the reader.
First of all, one can assume that the eigenvalues of $B_{11}$ are degenerate.
In addition, one can assume that the blocks $B_{11}$ and $B_{22}$ are diagonal matrices.  If not, one can simultaneously diagonalize them without 
changing the bound on the operator norms of the off-diagonal blocks.  Thus, $B_{11}=diag(\lambda_1, \lambda_2, \ldots, \lambda_{n_1}), $ and
$B_{22} =diag(\lambda_{n_1+1}, \ldots, \lambda_{n}).$   Then the eigenvectors of $B_{11}$ are given by $e_1, \ldots e_{n_1}$, and the eigenvectors
of $B_{22}$ are given by $e_{n_1+1}, \ldots e_{n}$, where $e_i, \ 1\leq i\leq n$ are the standard basis vectors in $\C^n.$  We recall that
$$ \lambda_1\leq \lambda_2 \leq \ldots \lambda_{n_1} < \lambda_{n_1+1}\leq \ldots \lambda_{n}, $$ and $\lambda_{n_1+1} - \lambda_{n_1}>Const.$
Then it is easy to see that
$$ \tilde{e_1}=e_1 +\sum_{j>n_1} \frac{\langle (B-\lambda_1) e_1, e_j\rangle}{\lambda_1-\lambda_j} \*e_j$$ is an approximate eigenvector of $B$ with the 
approximate eigenvalue $\lambda_1$ such that
$$(B-\lambda_1) \tilde{e_1}= \sum_{j>n_1} \frac{\langle (B-\lambda_1) e_1, e_j\rangle}{\lambda_1-\lambda_j} \*(B-\lambda_j) \* e_j.$$
Since $\|(B-\lambda_j) \* e_j \|\leq \epsilon, \ 1\leq j\leq n, $ and $\lambda_j-\lambda_1\geq Const, \ n_1<j\leq n,$ we obtain that
$$ \|(B-\lambda_1) \tilde{e_1}\| \leq const \* \epsilon^2,$$
where $const$ depends just on $Const$ and $n.$
The last inequality also holds (albeit with a different value of $const$) if one replaces $\tilde{e_1}$ by the normalized vector $\frac{\tilde{e_1}}
{\|\tilde{e_1}\|}.$ 
Thus, $B$ has an eigenvalue in the $const \* \epsilon^2$- neighborhood of $\lambda_1.$
\end{proof}
The result of the lemma can be immediately extended by induction to the case of $m\times m$ block matrices 
$B=(B_{ij})_{1\leq i,j\leq m}.$  To apply it in our setting, we note that
the $k_1\times k_1$ matrix $\frac{1}{\theta_1}\*Id -\frac{1}{\sqrt{N}} \* \Xi^{(m)}_N$ is the 
upper-left block of $Z_N(\rho_1).$  The other diagonal blocks of $Z_N(\rho_1)$   are given by 
$\Xi^{(i)}_N, \ 1\leq i \leq m-1 $ defined in (\ref{thetamatrica}).
Since the operator norms of 
the off-diagonal blocks of $Z_N(\rho_1)$ are $O(N^{-1/2}) (see (\ref{andor})),$ the desired statement follows.

Therefore, we have 
\begin{equation}
\label{saraevo}
g_{\sigma}(\rho_1) + \left(g_{\sigma}'(\rho_1) +o(1)\right)\* (x_i-\rho_1) = \frac{1}{\theta_1} - \frac{1}{\sqrt{N}}\* y_i 
+ O\left(\frac{\log (N)\*N^{1/6}}{N}\right), \ 1\leq i \leq k_1,
\end{equation}
where $y_1\geq \ldots \geq y_{k_1}$ are the eigenvalues of the matrix $\Xi^{(m)}_N.$
The result of Proposition \ref{Prop1} now follows from (\ref{saraevo}) and (\ref{tictac}).
\end{proof}
Since the eigenvalues of the matrix $\Xi^{(m)}_N(\rho_1)$ are bounded in probability, the first part of Theorem \ref{thm:fluctoutliers}, i.e.
(\ref{past}), follows from (\ref{uzheuzhe}) in Proposition \ref{Prop1}.

Now assume that the marginal distributions of the matrix entries of $W_N$ satisfy the Poincar\'e inequality (\ref{poin}) with a uniform constant
$\upsilon.$  Our goal is the almost sure bound 
(\ref{smeshno}) on the rate of convergence of the outliers.  
We note that one can improve (\ref{maroc}) and (\ref{tunis}) in Lemma \ref{Lemma14} as follows.
Applying (\ref{ska4}-\ref{ska5}) to (\ref{mnogo10}) and taking into account that $\xi_N(x)=\tilde{\xi}_N(x) $ and $\zeta_N(x)=\tilde{\zeta}_N(x)$ 
for all $x \in [2\sigma+2\delta, L]$  on a set of probability $1- O\left(\exp\left(-\frac{\sqrt{\upsilon\*N}\*\delta}{2}\right)\right), $
one proves
\begin{align}
\label{cvet1}
& \max \left(  | \zeta_N(x) |, \ x \in [2\*\sigma+2\*\delta, L] \right) \leq Const_1 \log (N), \\
\label{tunis1}
& \max \left( |z_i(x)-\frac{1}{\theta_1}|,\ x \in [2\*\sigma+2\*\delta, L] \right)   \leq Const_2\*\frac{\log(N)}{\sqrt{N}}, \ 1\leq i \leq k_1,
\end{align}
almost surely, where $Const_1>0, \ Const_2>0 $ are sufficiently large, improving (\ref{cvet}).
Reasoning as before,
(\ref{cvet1}) implies that
\begin{equation}
\label{maroc1}
|z_i(x)-z_i(y)| \leq Const_3 \* \frac{\log(N)}{\sqrt{N}}\*|x-y|, \ \forall x,y \in [2\*\sigma +2\delta, L], 
\ i=1, \ldots, r,
\end{equation}
almost surely for sufficiently large constant $Const_3>0.$
Thus, we have
\begin{equation}
\label{saraevo1}
g_{\sigma}(\rho_1) + \left(g_{\sigma}'(\rho_1) +o(1)\right)\* (x_i-\rho_1) = \frac{1}{\theta_1} - \frac{1}{\sqrt{N}}\* y_i 
+ O\left(\frac{\log (N)}{N}\right), \ 1\leq i \leq k_1, 
\end{equation}
almost surely, which implies (\ref{smeshno}) since $g_{\sigma}(\rho_1)=\frac{1}{\theta_1}.$
Theorem \ref{thm:fluctoutliers} is proven.
\end{proof}
\section{\bf{Proof of Theorems 1.3, 1.4, and 1.7}}
\label{sec:loc}
In this section, we prove Theorems \ref{thm:caseA}, \ref{thm:caseA1}, and \ref{thm:Local}.
We start with Theorem \ref{thm:caseA}.
\begin{proof}
Let $\theta_j>\sigma$ be an eigenvalue of $A_N$ with the multiplicity $k_j.$
Let us assume that Case A takes place. Thus, without loss of generality, we can assume that the eigenvectors of $A_N$ corresponding to the eigenvalue 
$\theta_j$ belong to $Span(e_1, \ldots, e_{K_j}), $ where $K_j$ is a fixed positive integer.  As always, we consider the real symmetric case.  
The treatment of the Hermitian case is very similar.
Consider a $K_j\times k_j$ matrix $U_j$
such that the ($K_j$-dimensional) columns of $U_j$ are filled by the first $K_j$ coordinates of the $k_j$ orthonormal vectors of $A_N$ corresponding to 
the eigenvalue $\theta_j.$  We recall that the remaining $N-K_j$ coordinates of these orthonormal vectors are zero.
Let us denote by $R_N^{(K_j)}(z)$ the upper-left $K_j\times K_j$  submatrix of the resolvent matrix $R_N(z)=(z\*I_N-X_N)^{-1}.$
Finally, we define the random matrix-valued field
\begin{equation}
\label{upsN}
\Upsilon_N(z):=\sqrt{N}\*\left(R_N^{(K_j)}(z)-g_{\sigma}(z)\*I_{K_j}\right), \ z\in \C \setminus [-2\sigma, 2\sigma].
\end{equation}
It follows from the definition that $\Upsilon_N(z)$ is a random function on $\C \setminus [-2\sigma, 2\sigma]$ with values in the space of complex 
symmetric $K_j\times K_j$ matrices.  In particular, $\Upsilon_N(x)$ is real symmetric for real $x \in \R \setminus [-2\sigma, 2\sigma].$
It follows from (\ref{thetamatrica}), (\ref{upsN}) and the definition of $U_j$ in the paragraph above (\ref{upsN}) that
\begin{equation}
\label{tor}
\Xi^{(j)}_N=  U_j^* \*\Upsilon_N(\rho_j) \* U_j.
\end{equation}
We recall that Theorem 1.1 in \cite{PRS} states that $\Upsilon_N(z)$ converges weakly in finite-dimensional distributions to a random field
\begin{equation}
\label{ups}
\Upsilon(z):= g^2_{\sigma}(z) \*\left(W^{(K_j)}+ Y(z)\right),
\end{equation}
where $W^{(K_j)}$ is the $K_j\times K_j$ upper-left corner submatrix of the Wigner matrix $W_N$ (\ref{offdiagreal1}-\ref{diagreal2}), and 
$Y(z)=(Y_{il}(z)), \ Y_{il}(z)=Y_{li}(z), \ 1\leq i,l\leq K_j, $ is a centered Gaussian random field with the covariance matrix given 
in (1.18)-(1.23) in \cite{PRS}.  In particular, for real $x \in \R \setminus [-2\sigma, 2\sigma],$ one has that the entries 
$Y_{il}, \ 1\leq i\leq l\leq K_j$ are independent centered Gaussian real random variables such that
\begin{align}
\label{Disp1} 
& \V(Y_{ii}(x))= \kappa_4(\mu)\* g_{\sigma}^2(x) - 2\*\sigma^4\* g_{\sigma}'(x), \ 1\leq i \leq K_j, \\
\label{Disp2}
& \V(Y_{il}(x))= -\sigma^4\* g_{\sigma}'(x), \ 1\leq i <l \leq K_j.
\end{align}
Now, Theorem \ref{thm:caseA} follows from Proposition \ref{Prop1} in this paper, and Theorem 1.1 in \cite{PRS}, since
\begin{equation}
g_{\sigma}(\rho_j)=\frac{1}{\theta_j}, \ \text{and}  \ \ g_{\sigma}'(\rho_j)=  -\frac{1}{\theta_j^2 - \sigma^2}. 
\end{equation}
Theorem \ref{thm:caseA} is proven. The proof of Theorem \ref{thm:caseA1} is very similar to the given proof of Theorem \ref{thm:caseA}.  One has to use
Theorems 4.1 and 4.2 and Remark 4.1 in \cite{ORS} that generalize Theorems 1.1 and 1.5 in \cite{PRS} to the non-i.i.d. case, and replace
$\kappa_4(\mu)$ in (\ref{Disp1}) with $\kappa_4(i), \ 1\leq i \leq K_j.$
\end{proof}
Now, we turn to the proof of Theorem \ref{thm:Local}.
\begin{proof}
Recall that we extended the definition of $\Upsilon(z) $ 
to that of an infinite-dimensional  matrix $ \Upsilon(z)_{pq}, \ 1\leq p,q <\infty, \ $ in the paragraphs above the formulation of
Theorem \ref{thm:Local}.
We employ a standard approximation argument.  For simplicity, we assume that $\Im z \not=0.\ $  If $z$ is real, one has to replace $R_N(z)$ by
$h(X_N)\* R_N(z),$ where $h \in C^{\infty}_c(\R)$ is defined in (\ref{h1}-\ref{h2}).
Let $n$ be a sufficiently large fixed positive integer and consider $N\geq n.$
It follows from Theorem 1.1 in \cite{PRS} (see the proof of Theorem \ref{thm:caseA} above) that
the joint distribution of 
$$\langle u^{(n)}_p, \Upsilon_N(z) u^{(n)}_q \rangle= 
\sqrt{N}\*\left(\langle u^{(n)}_p, R_N(z) u^{(n)}_q \rangle- g_\sigma(z)\* \langle u^{(n)}_p, u^{(n)}_q \rangle \right), \ \ 1\leq p,q\leq l, $$
converges weakly to the joint distribution of
$ \ \langle u^{(n)}_p, \Upsilon(z) u^{(n)}_q \rangle, \ \ 1\leq p,q\leq l.$
Choosing $n$  sufficiently large, we can make
\begin{equation}
\label{torpedozil}
\V \left(\langle u^{(n)}_p, \Upsilon(z) u^{(n)}_q \rangle - \langle u_p, \Upsilon(z) u_q \rangle \right), \ \ 1\leq p,q\leq l,
\end{equation}
and
\begin{equation}
\label{torpedo}
\V \left(\langle u^{(n)}_p, \Upsilon_N(z) u^{(n)}_q \rangle - \langle u^{(N)}_p, \Upsilon(z)_N u^{(N)}_q \rangle \right), \ \ 1\leq p,q\leq l,
\end{equation}
arbitrary small uniformly in $N\geq n.$
Indeed, the variance in (\ref{torpedozil}) is bounded by $ O(\|u_p-u^{(n)}_p\|^2 +\|u_q-u^{(n)}_q\|^2) \ $ since the entries of $\Upsilon(z)$ are 
i.i.d random variables with bounded variance on the diagonal and i.i.d. random variables with bounded variance off the diagonal.
In addition,
\begin{align}
& \langle u^{(n)}_p, \Upsilon_N(z) u^{(n)}_q \rangle - \langle u^{(N)}_p, \Upsilon_N(z) u^{(N)}_q \rangle=
\langle u^{(n)}_p - u^{(N)}_p, \Upsilon_N(z) u^{(n)}_q \rangle \nonumber \\ 
& + \langle u^{(N)}_p, \Upsilon_N(z) (u^{(n)}_q -u^{(N)}_q) \rangle \nonumber
\end{align}
and we can use the bounds (\ref{ska1}) and (\ref{ska2}) in Theorem  \ref{thm:quadratic} rewritten as
\begin{align}
\label{ska11}
& \V \left(\langle u, f(X_N) \* v \rangle \right) 
\leq const_1 \* \frac{\|f\|_{C^5_c(\R)} \*\|u\|^2 \* \|v\|^2}{N}, \\
& \label{ska21}
\E \langle u, f(X_N) \* v \rangle - \langle u, v \rangle \* \int_{-2\sigma}^{2\sigma} f(x) \* d \mu_{sc}(dx)|
\leq const_2 \* \|f\|_{C^8_c(\R)} \*\frac{\|u\|\*\|v\|}{\sqrt{N}},
\end{align}
to show that 
\begin{equation*}
\V\left(\langle u^{(n)}_p - u^{(N)}_p, \Upsilon_N(z) u^{(n)}_q \rangle\right), 
\ \V\left(\langle u^{(N)}_p, \Upsilon_N(z) (u^{(n)}_q -u^{(N)}_q) \rangle\right)
\end{equation*}
are arbitrary small (uniformly in $N$) provided one chooses $n$ 
sufficiently large.  This finishes the proof.
\end{proof}
Theorem \ref{thm:Local} allows the following extension of Theorem \ref{thm:caseA}:
\begin{remark}
\label{rem}
Let $u^{(1)}, \ldots, u^{(r)} \ $ be a set of orthonormal vectors in $ l^2(\N) $ such that 
\begin{equation}
\label{pochtiort}
\|u^{(p)}-u^{(p)}_N \|=o(N^{-1/2}), \ 1\leq p\leq r,
\end{equation}
where $u^{(p)}_N $ denotes the projection of $u^{(p)}$ onto the subspace spanned by the first $N$ standard basis vectors $e_1, \ldots, e_N. $
Let $U_N$ be the $N\times r$ matrix whose columns are given by the vectors $u^{(1)}_N, \ldots, u^{(r)}_N.\ $ Also denote by $\Theta $ the $r\times r $
diagonal matrix 
$$ \Theta=diag(\theta_1, \ldots, \theta_1, \ldots, \theta_{j_0-1}, \ldots,  \theta_{j_0-1}, \theta_{j_0+1}, \ldots,  \theta_{j_0+1}, \ldots
\theta_J, \ldots, \theta_J).$$
Finally, define 
$$A_N=U_N\*\Theta \*U_N^*.$$
The result of Theorem \ref{thm:caseA} can be extended for such $A_N$, with the matrix $V_j $ given by 
\begin{equation}
\label{strahov}
\theta_j^2\*\langle u^{(p)}, \Upsilon(\theta_j)\* u^{(q)} \rangle, \ \ 1\leq p,q \leq k_j.
\end{equation}
\end{remark}

\section{\bf{Appendix}}
\label{section:appendix}
The appendix contains several basic formulas used throughout the paper.  

First, we recall the decoupling formula from \cite{KKP}.
Let $\xi$ be a real-valued random variable with $p+2$ finite moments, and $\phi$ be a function from $\C \to \R$ with $p+1$ continuous and bounded 
derivatives.  Then 
\begin{equation}
\label{decouple} 
 \E(\xi \phi(\xi)) = \sum_{a=0}^p \frac{\kappa_{a+1}}{a!} \E(\phi^{(a)}(\xi)) + \epsilon,  
 \end{equation}
where $\kappa_a$ are the cumulants of $\xi$, $|\epsilon| \leq C \sup_t \big| \phi^{(p+1)}(t) \big| \E(|\xi|^{p+2})$, and $C$ depends only on $p$. If
$\xi$ is a centered Gaussian random variable, the decoupling formula (\ref{decouple}) becomes
\begin{equation}
\label{decoupleGauss} 
 \E(\xi \phi(\xi))= \V(\xi)\* \E(\phi'(\xi)),
\end{equation}
and can be immediately verified by integration by parts.

Next, we write a basic resolvent identity.
For any two Hermitian matrices $X_1$ and $X_2$ and non-real $z$ we have:
\begin{equation}
\label{resident}
 (zI - X_2)^{-1} = (zI - X_1)^{-1} - (zI - X_1)^{-1}(X_1 - X_2) (zI - X_2)^{-1}  
\end{equation}
As a corollary of (\ref{resident}), one has the following formulas.
If $X$ is a real symmetric matrix with resolvent $R$ then  
\begin{align}
\label{vecher1}
& \frac{\partial R_{kl}}{\partial X_{pq}} = R_{kp}\*R_{ql} +R_{kq} \*R_{pl}, \ \text{for} \ p \not = q, \\
\label{vecher2}
& \frac{\partial R_{kl}}{\partial X_{pp}} = R_{kp}\*R_{pl} .
\end{align}
In a similar way, if $X$ is a Hermitian matrix then 
\begin{align}
& \frac{\partial R_{kl}}{\partial \Re X_{pq}} = R_{kp}\*R_{ql} +R_{kq} \*R_{pl}, \  p \not = q, \nonumber \\
& \frac{\partial R_{kl}}{\partial \Im X_{pq}} = i\*\left(R_{kp}\*R_{ql} - R_{kq}\*R_{pl} \right), \  p \not = q,  \nonumber \\
& \frac{\partial R_{kl}}{\partial X_{pp}} = R_{kp}\*R_{pl}. \nonumber
\end{align}

Finally, we will use the following properties of the resolvent:
\begin{equation}
\label{norma17}
\| R_N(z) \|=\frac{1}{dist(z, Sp(X))},
\end{equation}
where by $Sp(X) $ we denote the spectrum of a real symmetric (Hermitian) matrix $X.$
The bound (\ref{norma17}) implies
\begin{equation}
\label{resbound}
\| R_N(z) \| \leq |\Im(z)|^{-1}. 
\end{equation}
Therefore, all entries of the resolvent matrix are bounded by $|\Im(z)|^{-1}$.
In a similar fashion, we have the following bound for the Stieltjes transform, $g(z)$, of any probability measure:
\begin{equation}
\label{STbound}
 | g(z) | \leq |\Im(z)|^{-1} 
\end{equation}

\end{document}